\documentclass[a4paper]{amsart}
\usepackage[foot]{amsaddr}

\usepackage[latin1]{inputenc}
\usepackage{amssymb,amsmath,amsthm}
\usepackage{amsmath}
\usepackage{amsfonts}
\usepackage{enumitem}
\usepackage{booktabs}
\usepackage{ifthen}
\usepackage{tikz}
\usetikzlibrary{math,calc}
\usepackage{url}
\usepackage{hyphenat}
\usepackage{hyperref}
\usepackage{graphicx}
\usepackage{calc,etoolbox}
\usepackage{rotating}
\usepackage[para]{threeparttable}
\usepackage{caption}
\usepackage{multicol}
\usepackage{multirow}
\usepackage{subfig}
\usepackage{caption}

\theoremstyle{plain}
\newtheorem{theorem}{Theorem}[section]
\newtheorem{proposition}[theorem]{Proposition}
\newtheorem{lemma}[theorem]{Lemma}

\theoremstyle{definition}
\newtheorem{definition}[theorem]{Definition}
\newtheorem{example}[theorem]{Example}
\newtheorem{remark}[theorem]{Remark}

\newtheorem{claim}{Claim}
\numberwithin{claim}{theorem}
\newenvironment{pfclaim}[1][Proof]{\begin{trivlist}
\item[\hskip \labelsep {\textit{{#1}.}}]} {{\footnotesize $\blacksquare$}\end{trivlist}}

\numberwithin{figure}{section}
\numberwithin{table}{section}

\newcommand{\IN}{\ensuremath{\mathbb{N}}}

\newcommand{\minorant}{\ensuremath{\leqq}}

\newcommand{\nset}[1]{\ensuremath{[{#1}]}}
\newcommand{\card}[1]{\ensuremath{\lvert{#1}\rvert}}

\newcommand{\gen}[2][\gendefault]{\ensuremath{\langle{#2}\rangle_{#1}}}
\newcommand{\clonegen}[1]{\gen[]{#1}}
\newcommand{\vect}[1]{\ensuremath{\mathbf{#1}}}

\newcommand{\clIntVal}[3]{\ensuremath{#1_{\ifthenelse{\equal{#2}{}}{\mathord{*}}{#2}\ifthenelse{\equal{#3}{}}{\mathord{*}}{#3}}}}

\newcommand{\clIntNeq}[1]{\ensuremath{#1_{\neq}}}
\newcommand{\clIntLeq}[1]{\ensuremath{#1_{\leq}}}
\newcommand{\clIntGeq}[1]{\ensuremath{#1_{\geq}}}
\newcommand{\clIntNeqOO}[1]{\ensuremath{#1_{{\neq},00}}}
\newcommand{\clIntNeqII}[1]{\ensuremath{#1_{{\neq},11}}}

\newcommand{\clAll}{\ensuremath{\mathsf{\Omega}}}
\newcommand{\clEioo}{\ensuremath{\clIntNeqII{\clAll}}}
\newcommand{\clEioi}{\ensuremath{\clIntGeq{\clAll}}}
\newcommand{\clEiio}{\ensuremath{\clIntLeq{\clAll}}}
\newcommand{\clEiii}{\ensuremath{\clIntNeqOO{\clAll}}}

\newcommand{\clNeq}{\ensuremath{\clIntNeq{\clAll}}}
\newcommand{\clOX}{\ensuremath{{\clIntVal{\clAll}{0}{}}}}
\newcommand{\clIX}{\ensuremath{{\clIntVal{\clAll}{1}{}}}}
\newcommand{\clXO}{\ensuremath{{\clIntVal{\clAll}{}{0}}}}
\newcommand{\clXI}{\ensuremath{{\clIntVal{\clAll}{}{1}}}}
\newcommand{\clOXC}{\ensuremath{\clOX \cup \clVak}}
\newcommand{\clOXCI}{\ensuremath{\clOX \cup \clVaki}}
\newcommand{\clIXC}{\ensuremath{\clIX \cup \clVak}}

\newcommand{\clXOC}{\ensuremath{\clXO \cup \clVak}}
\newcommand{\clXOCI}{\ensuremath{\clXO \cup \clVaki}}
\newcommand{\clXIC}{\ensuremath{\clXI \cup \clVak}}

\newcommand{\clOO}{\ensuremath{\clIntVal{\clAll}{0}{0}}}
\newcommand{\clII}{\ensuremath{\clIntVal{\clAll}{1}{1}}}
\newcommand{\clOI}{\ensuremath{\clIntVal{\clAll}{0}{1}}}

\newcommand{\clS}{\ensuremath{\mathsf{S}}}
\newcommand{\clSc}{\ensuremath{\clIntVal{\clS}{0}{1}}}

\newcommand{\clSM}{\ensuremath{\mathsf{SM}}}

\newcommand{\clSmin}{\ensuremath{\clS^{-}}}
\newcommand{\clSmaj}{\ensuremath{\clS^{+}}}

\newcommand{\clM}{\ensuremath{\mathsf{M}}}
\newcommand{\clMo}{\ensuremath{\clIntVal{\clM}{0}{}}}
\newcommand{\clMi}{\ensuremath{\clIntVal{\clM}{}{1}}}
\newcommand{\clMc}{\ensuremath{\clIntVal{\clM}{0}{1}}}

\newcommand{\clUk}[1]{\ensuremath{\mathsf{U}^{#1}}}

\newcommand{\clTcUk}[1]{\ensuremath{\clIntVal{\clUk{#1}}{0}{1}}}

\newcommand{\clMUk}[1]{\ensuremath{\mathsf{M}\clUk{#1}}}
\newcommand{\clMcUk}[1]{\ensuremath{\clIntVal{\clMUk{#1}}{0}{1}}}
\newcommand{\clWk}[1]{\ensuremath{\mathsf{W}^{#1}}}

\newcommand{\clTcWk}[1]{\ensuremath{\clIntVal{\clWk{#1}}{0}{1}}}

\newcommand{\clMWk}[1]{\ensuremath{\mathsf{M}\clWk{#1}}}
\newcommand{\clMcWk}[1]{\ensuremath{\clIntVal{\clMWk{#1}}{0}{1}}}

\newcommand{\clRefl}{\ensuremath{\mathsf{R}}}  
\newcommand{\clReflOO}{\ensuremath{\clIntVal{\clRefl}{0}{0}}}

\newcommand{\clVak}{\ensuremath{\mathsf{C}}}
\newcommand{\clVaka}[1]{\ensuremath{\clVak_{#1}}}

\newcommand{\clVako}{\ensuremath{\clVaka{0}}}
\newcommand{\clVaki}{\ensuremath{\clVaka{1}}}

\newcommand{\clKlik}[2]{\ensuremath{\mathsf{U}^{#1}_{#2}}}

\newcommand{\clL}{\ensuremath{\mathsf{L}}}
\newcommand{\clLo}{\ensuremath{\clIntVal{\clL}{0}{}}}
\newcommand{\clLi}{\ensuremath{\clIntVal{\clL}{}{1}}}
\newcommand{\clLc}{\ensuremath{\clIntVal{\clL}{0}{1}}}
\newcommand{\clLS}{\ensuremath{\mathsf{LS}}}
\newcommand{\clLambda}{\ensuremath{\mathsf{\Lambda}}}
\newcommand{\clV}{\ensuremath{\mathsf{V}}}
\newcommand{\clLambdac}{\ensuremath{\clIntVal{\clLambda}{0}{1}}}
\newcommand{\clVc}{\ensuremath{\clIntVal{\clV}{0}{1}}}
\newcommand{\clLambdao}{\ensuremath{\clIntVal{\clLambda}{0}{}}}
\newcommand{\clVo}{\ensuremath{\clIntVal{\clV}{0}{}}}
\newcommand{\clLambdai}{\ensuremath{\clIntVal{\clLambda}{}{1}}}
\newcommand{\clVi}{\ensuremath{\clIntVal{\clV}{}{1}}}
\newcommand{\clOmegaOne}{\ensuremath{\clAll(1)}}
\newcommand{\clIstar}{\ensuremath{\mathsf{I}^{*}}}

\newcommand{\clI}{\ensuremath{\mathsf{I}}}
\newcommand{\clIo}{\ensuremath{\mathsf{I}_0}}
\newcommand{\clIi}{\ensuremath{\mathsf{I}_1}}
\newcommand{\clIc}{\ensuremath{\mathsf{J}}}

\newcommand{\clAff}{\ensuremath{\mathsf{A}}}
\newcommand{\clICD}[1]{\ensuremath{\mathsf{D}^{#1}}}

\newcommand{\vak}[1]{\ensuremath{\mathrm{c}_{#1}}}
\newcommand{\id}{\ensuremath{\mathrm{id}}}

\newcommand{\Pii}[1]{\ensuremath{\mathcal{P}_{1,2}(#1)}}
\newcommand{\Pone}[1]{\ensuremath{\mathcal{P}_{1}(#1)}}

\DeclareMathOperator{\pr}{pr}

\newcommand{\closys}[1]{\ensuremath{\mathcal{L}_{#1}}}
\newcommand{\clProj}[1]{\ensuremath{\mathsf{J}_{#1}}}

\DeclareMathOperator{\codim}{codim}
\DeclareMathOperator{\icodim}{\text{$\cap$}-codim}

\makeatletter
\newcommand{\displaybump}{\hbox to \@totalleftmargin{\hfil}}
\makeatother

\begin{document}
\title[Clonoids of Boolean functions]{Clonoids of Boolean functions with a linear source clone and a semilattice or 0- or 1-separating target clone}

\author{Erkko Lehtonen}

\address{%
    Department of Mathematics \\
    Khalifa University of Science and Technology \\
    P.O. Box 127788 \\
    Abu Dhabi \\
    United Arab Emirates
}

\date{\today}

\begin{abstract}
Extending Sparks's theorem, we determine the cardinality of the lattice of $(C_1,C_2)$\hyp{}clonoids of Boolean functions for certain pairs $(C_1,C_2)$ of clones of Boolean functions.
Namely, when $C_1$ is a subclone (a proper subclone, resp.)\ of the clone of all linear (affine) functions and $C_2$ is a subclone of the clone generated by a semilattice operation and constants (a subclone of the clone of all $0$- or $1$-separating functions, resp.), then the lattice of $(C_1,C_2)$\hyp{}clonoids is uncountable.
Combining this fact with several earlier results, we obtain a complete classification of the cardinalities of the lattices of $(C_1,C_2)$\hyp{}clonoids for all pairs $(C_1,C_2)$ of clones on $\{0,1\}$.
\end{abstract}

\maketitle


\section{Introduction}

This paper belongs to the field of universal algebra, focusing on the theory of clones, clonoids, and related topics.
Its key notion is that of function class composition, which is defined as follows.
If $F$ and $G$ are sets of multivariate functions, then the composition of $F$ with $G$, denoted $FG$,
is the set of all well\hyp{}defined composite functions of the form $f(g_1, \dots, g_n)$, where $f \in F$ and $g_1, \dots, g_n \in G$.

The $i$\hyp{}th $n$\hyp{}ary \emph{projection} on $A$ is the operation $(a_1, \dots, a_n) \mapsto a_i$.
The set of all projections on $A$ is denoted by $\clProj{A}$.
A \emph{clone} on $A$ is a set $C$ of operations on $A$ that contains all projections on $A$ and is closed under composition, that is, $\clProj{A} \subseteq C$ and $C C \subseteq C$.
For fixed clones $C_1$ and $C_2$ on sets $A$ and $B$, respectively, a \emph{$(C_1,C_2)$\hyp{}clonoid} is a set $F$ of functions of several arguments from $A$ to $B$ satisfying $F C_1 \subseteq F$ and $C_2 F \subseteq F$.
As a special case of this, we have the $(\clProj{A},\clProj{B})$\hyp{}clonoids, which are usually called \emph{minions} or \emph{minor\hyp{}closed classes}.

In universal algebra,
clones arise in a natural way as sets of term operations of algebras or as sets of polymorphisms of relations.
Minions, in turn, arise as sets of term operations induced by terms of height $1$ or as sets of polymorphisms of relation pairs (see Pippenger~\cite{Pippenger}).
We obtain $(C_1,C_2)$\hyp{}clonoids as sets of polymorphisms of relation pairs $(R,S)$, where $R$ and $S$ are invariants of $C_1$ and $C_2$, respectively (see Couceiro and Foldes \cite{CouFol-2005,CouFol-2009}).

In theoretical computer science, universal\hyp{}algebraic tools, including clones and polymorphisms, have proved useful in the analysis of computational complexity of constraint satisfaction problems (CSP).
In particular, minions arise in the context of a new variant known as promise CSP\@.
For further information, see the survey article by Barto et al.\ \cite{BarBulKroOpr}.

To the best of the author's knowledge, the term ``clone'' was first used in the universal\hyp{}algebraic sense in the 1965 monograph of Cohn~\cite{Cohn}, who attributed it to Philip Hall.
The term ``clonoid'' was introduced in the 2016 paper by Aichinger and Mayr \cite{AicMay}, and ``minion'' was coined by Opr\v{s}al around the year 2018 (see \cite[Definition~2.20]{BarBulKroOpr}, \cite{BulKroOpr}).
It should, however, be noted that these concepts have appeared in the literature much earlier under different names.
For further information and general background on universal algebra and clones, see, e.g., the monographs by Bergman~\cite{Bergman} and Szendrei~\cite{Szendrei}.

Classification of clones on finite sets with at least three elements is one of the main open problems in universal algebra, and numerous studies have been published on this topic.
In complete analogy with the classification problem of clones, we are led to the classification problem of $(C_1,C_2)$\hyp{}clonoids.
Such classification results for certain types of clone pairs $(C_1,C_2)$ have appeared in the literature,
for example, in the works of
Fioravanti~\cite{Fioravanti-AU,Fioravanti-IJAC},
Kreinecker~\cite{Kreinecker},
and
Mayr and Wynne~\cite{MayWyn}.

We would like to highlight the following remarkable result due to Sparks, which can be viewed as a starting point for our work.
Here, $\closys{(C_1,C_2)}$ denotes the lattice of $(C_1,C_2)$\hyp{}clonoids.
An operation $f \colon A^n \to A$ is a \emph{near\hyp{}unanimity operation} if it satisfies all identities of the form $f(x, \dots, x, y, x, \dots, x) \approx x$, where the single $y$ is at any argument position.
A ternary near\hyp{}unanimity operation is called a \emph{majority operation}.
A \emph{Mal'cev operation} is a ternary operation that satisfies the identities $f(x,x,y) \approx f(y,x,x) \approx y$.

\begin{theorem}[{Sparks~\cite[Theorem~1.3]{Sparks-2019}}]
\label{thm:Sparks}
Let $A$ be a finite set with $\card{A} > 1$, and let $B = \{0,1\}$.
Let $C$ be a clone on $B$.
Then the following statements hold.
\begin{enumerate}[label={\upshape(\roman*)}]
\item\label{thm:Sparks:finite} $\closys{(\clProj{A},C)}$ is finite if and only if $C$ contains a near\hyp{}unanimity operation.
\item\label{thm:Sparks:countable} $\closys{(\clProj{A},C)}$ is countably infinite if and only if $C$ contains a Mal'cev operation but no majority operation.
\item\label{thm:Sparks:uncountable} $\closys{(\clProj{A},C)}$ has the cardinality of the continuum if and only if $C$ contains neither a near\hyp{}unanimity operation nor a Mal'cev operation.
\end{enumerate}
\end{theorem}

Considering that the clones on the two\hyp{}element set $\{0,1\}$ are well known (see Subsection~\ref{subsec:Post} and Figure~\ref{fig:Post}), the present author initiated the effort of systematically counting and enumerating all $(C_1,C_2)$\hyp{}clonoids, for each pair $(C_1,C_2)$ of clones on $\{0,1\}$.
This was achieved for various pairs $(C_1,C_2)$ of clones in a series of papers \cite{CouLeh-Lcstability,Lehtonen-SM,Lehtonen-nu,Lehtonen-discmono,Lehtonen-ess-lin-sem-sep} by the present author, partly in collaboration with M. Couceiro.
See \cite[Section 7, Table~7.1]{Lehtonen-ess-lin-sem-sep} for a summary of the results of these previous papers.

However, there still remain a few pairs $(C_1,C_2)$ of clones on $\{0,1\}$ for which the cardinality of the clonoid lattice $\closys{(C_1,C_2)}$ and the $(C_1,C_2)$\hyp{}clonoids are unknown.
This is the case when
the source clone $C_1$ belongs to the interval between the clone $\clLc$ of idempotent linear functions and the clone $\clL$ of all linear functions
and
the target clone $C_2$ belongs to the interval between the clone $\clLambdac$ of conjunctions and the clone $\clUk{\infty}$ of $1$\hyp{}separating functions of rank $\infty$ or, dually, to the interval between the clone $\clVc$ of disjunctions and the clone $\clWk{\infty}$ of $0$\hyp{}separating functions of rank $\infty$.
Nevertheless, in the case when $C_1 = \clL$ and $C_2$ contains $\clMcUk{\infty}$ or $\clMcWk{\infty}$ (monotone, idempotent, $1$- or $0$\hyp{}separating functions of rank $\infty$), we know that $\closys{(C_1,C_2)}$ is finite.

It is the purpose of this paper to fill this gap and to determine the cardinality of the $(C_1,C_2)$\hyp{}clonoid lattice in the remaining cases described in the previous paragraph.
It turns out that in all these cases, the $(C_1,C_2)$\hyp{}clonoid lattice is uncountable (Theorem~\ref{thm:Ulin}).
The proof presented in Section~\ref{sec:uncountable} is based on the idea of constructing a countably infinite family $F$ of Boolean functions with the property that for every subset $S$ of $F$, the $(C_1,C_2)$\hyp{}clonoid generated by $S$ contains no members of $F \setminus S$.
Consequently, distinct subsets of $F$ always generate distinct $(C_1,C_2)$\hyp{}clonoids.
Because the power set of a countably infinite set is uncountable, it follows that the lattice of $(C_1,C_2)$\hyp{}clonoids is uncountable.

While it may be infeasible to explicitly describe all $(C_1,C_2)$\hyp{}clonoids when there are an uncountable infinitude of them,
we present a few examples of such $(C_1,C_2)$\hyp{}clonoids in Section~\ref{sec:some}.
The examples are of two types.
Firstly, the $(C'_1,C'_2)$\hyp{}clonoids are known for certain superclones $C'_1 \supseteq C_1$ and $C'_2 \supseteq C_2$ from our earlier results; these are also $(C_1,C_2)$\hyp{}clonoids by the monotonicity of function class composition.
Secondly,
the concept of $(\clL,\clLambdac)$\hyp{}clonoid can be reformulated in the language of linear algebra, and we may use some basic notions from linear algebra to define new types of $(\clL,\clLambdac)$\hyp{}clonoids.

This paper brings to a conclusion our endeavour to determine the cardinality of the lattice of $(C_1,C_2)$\hyp{}clonoids for arbitrary source and target clones on $\{0,1\}$.
In this way, we have obtained a noteworthy extension of Sparks's theorem (Theorem~\ref{thm:Sparks}), which we state as our new Theorem~\ref{thm:card}.
Its proof, presented in Section~\ref{sec:summary}, essentially puts together all our earlier results on $(C_1,C_2)$\hyp{}clonoids of Boolean functions, and we provide exact references to the literature.

We conclude the paper with some final remarks and open problems in Section~\ref{sec:remarks}.


\section{Preliminaries}

\subsection{General}

The sets of nonnegative integers and positive integers are denoted by $\IN$ and $\IN^{+}$, respectively.
For $n \in \IN$, let $\nset{n} := \{ \, i \in \IN^{+} \mid 1 \leq i \leq n \, \}$. 
We denote tuples by bold letters and their components by the corresponding italic letters, e.g., $\vect{a} = (a_1, \dots, a_n)$.

\subsection{Clones and clonoids}

Let $A$ and $B$ be nonempty sets, and let $n \in \IN^{+}$.
An \emph{$n$\hyp{}ary function} from $A$ to $B$ is a mapping $f \colon A^n \to B$; here, the integer $n$ is called the \emph{arity} of $f$.
We denote by $\mathcal{F}_{AB}^{(n)}$ the set of all $n$\hyp{}ary functions from $A$ to $B$ and we let $\mathcal{F}_{AB} := \bigcup_{n \in \IN^{+}} \mathcal{F}_{AB}^{(n)}$ be the set of all \emph{finitary functions} from $A$ to $B$.
In the special case when $A = B$, we speak of \emph{operations} on $A$, and we write $\mathcal{O}_A^{(n)}$ and $\mathcal{O}_A$ for $\mathcal{F}_{AA}^{(n)}$ and $\mathcal{F}_{AA}$, respectively.
For $C \subseteq \mathcal{F}_{AB}$ and $n \in \IN$, the \emph{$n$\hyp{}ary part} of $C$ is $C^{(n)} := C \cap \mathcal{F}_{AB}^{(n)}$.

For $n, i \in \IN$ with $1 \leq i \leq n$, the $i$\hyp{}th $n$\hyp{}ary \emph{projection} on $A$ is the operation $\pr_i^{(n)} \colon A^n \to A$, $(a_1, \dots, a_n) \mapsto a_i$.
The set of all projections on $A$ is denoted by $\clProj{A}$.

For $f \in \mathcal{F}_{BC}^{(n)}$ and $g_1, \dots, g_n \in \mathcal{F}_{AB}^{(m)}$, the \emph{composition} of $f$ with $(g_1, \dots, g_n)$, denoted by $f(g_1, \dots, g_n)$, is a function in $\mathcal{F}_{AC}^{(m)}$ and is defined by the rule
\[
f(g_1, \dots, g_n)(\vect{a}) :=
f(g_1(\vect{a}), \dots, g_n(\vect{a})),
\]
for all $\vect{a} \in A^m$.

The notion of composition extends in a natural way to function classes.
For $F \subseteq \mathcal{F}_{BC}$ and $G \subseteq \mathcal{F}_{AB}$, the \emph{composition} of $F$ with $G$
is
\[
FG :=
\{ \, h \in \mathcal{F}_{AC} \mid
\exists m, n \in \IN^{+} \,
\exists f \in F^{(n)} \,
\exists g_1, \dots, g_n \in G^{(m)} \colon
h = f(g_1, \dots, g_n) \, \}
.
\]
It follows immediately from the definition that function class composition is monotone with respect to subset inclusion, i.e., if $F, F' \subseteq \mathcal{F}_{BC}$, $G, G' \subseteq \mathcal{F}_{AB}$ satisfy $F \subseteq F'$ and $G \subseteq G'$, then $F G \subseteq F' G'$.

Function class composition allows us to define our main objects of study in a convenient way.
A \emph{clone} on $A$ is a set $C \subseteq \mathcal{O}_A$ that is closed under composition and contains all projections on $A$, in symbols, $C C \subseteq C$ and $\clProj{A} \subseteq C$.
The clones on $A$ form a closure system on $\mathcal{O}_A$, and the clone generated by a set $F \subseteq \mathcal{O}_A$, i.e., the least clone containing $F$, is denoted by $\clonegen{F}$.

Let $A$ and $B$ be arbitrary nonempty sets, and let $C_1$ be a clone on $A$ (the \emph{source clone}) and let $C_2$ be a clone on $B$ (the \emph{target clone}).
A set $K \subseteq \mathcal{F}_{AB}$ is called a \emph{$(C_1,C_2)$\hyp{}clonoid} if $K C_1 \subseteq K$ and $C_2 K \subseteq K$
($K$ is stable under right composition with $C_1$ and under left composition with $C_2$).
The set $\closys{(C_1,C_2)}$ of all $(C_1,C_2)$\hyp{}clonoids forms a closure system on $\mathcal{F}_{AB}$,
and the $(C_1,C_2)$\hyp{}clonoid generated by a set $F \subseteq \mathcal{F}_{AB}$, i.e.,
the least $(C_1,C_2)$\hyp{}clonoid containing $F$, is denoted by $\gen[(C_1,C_2)]{F}$.
A $(\clProj{A},\clProj{B})$\hyp{}clonoid is called a \emph{minion} or a \emph{minor\hyp{}closed class}.

\subsection{Boolean functions}
\label{subsec:Bf}

Operations on $\{0,1\}$ are called \emph{Boolean functions}.
Table~\ref{tab:Bfs} defines a few well\hyp{}known Boolean functions:
$\vak{0}$ and $\vak{1}$ (constant functions),
$\id$ (identity),
$\neg$ (negation),
$\wedge$ (conjunction),
$\vee$ (disjunction),
$+$ (addition modulo $2$).
Recall that $\pr_i^{(n)}$ denotes the $i$\hyp{}th $n$\hyp{}ary projection;
thus $\id = \pr_1^{(1)}$.
For $1 \leq i \leq n$, we also let $\neg_i^{(n)} := \neg(\pr_i^{(n)})$, the $i$\hyp{}th $n$\hyp{}ary \emph{negation} (or \emph{negated projection}).
For $b \in \{0,1\}$ and $n \in \IN^{+}$, the $n$\hyp{}ary \emph{constant function} taking value $b$ is $\vak{b}^{(n)} \colon \{0,1\}^n \to \{0,1\}$, $\vak{b}^{(n)}(\vect{a}) = b$ for all $\vect{a} \in \{0,1\}^n$.
We will omit the superscript indicating the arity when it is clear from the context or irrelevant.

\begin{table}
\begin{tabular}[t]{c|cccc}
$x_1$ & $\vak{0}$ & $\vak{1}$ & $\id$ & $\neg$ \\
\hline
$0$ & $0$ & $1$ & $0$ & $1$ \\
$1$ & $0$ & $1$ & $1$ & $0$
\end{tabular}
\qquad\qquad
\begin{tabular}[t]{cc|ccc}
$x_1$ & $x_2$ & $\wedge$ & $\vee$ & $+$ \\
\hline
$0$ & $0$ & $0$ & $0$ & $0$ \\
$0$ & $1$ & $0$ & $1$ & $1$ \\
$1$ & $0$ & $0$ & $1$ & $1$ \\
$1$ & $1$ & $1$ & $1$ & $0$
\end{tabular}
\caption{Some well\hyp{}known Boolean functions.}
\label{tab:Bfs}
\end{table}

The \emph{complement} of $a \in \{0,1\}$ is $\overline{a} := \neg(a)$.
The \emph{complement} of $\vect{a} = (a_1, \dots, a_n) \in \{0,1\}^n$ is $\overline{\vect{a}} := (\overline{a_1}, \dots, \overline{a_n})$.
We regard the set $\{0,1\}$ totally ordered by the natural order $0 < 1$, which induces the direct product order on $\{0,1\}^n$.
The poset $(\{0,1\}^n, \mathord{\leq})$ is a Boolean lattice, i.e., a complemented distributive lattice with least and greatest elements $\vect{0} = (0, \dots, 0)$ and $\vect{1} = (1, \dots, 1)$ and with the map $\vect{a} \mapsto \overline{\vect{a}}$ being the complementation.

We will denote the set of all Boolean functions by $\clAll$, that is, $\clAll = \mathcal{O}_{\{0,1\}}$.
For $a, b \in \{0,1\}$, let
\begin{align*}
\clIntVal{\clAll}{a}{} &:= \{ \, f \in \clAll \mid f(0, \dots, 0) = a \, \}, \\
\clIntVal{\clAll}{}{b} &:= \{ \, f \in \clAll \mid f(1, \dots, 1) = b \, \}, \\
\clIntVal{\clAll}{a}{b} &:= \clIntVal{\clAll}{a}{} \cap \clIntVal{\clAll}{}{b}.
\end{align*}
Moreover, for any $K \subseteq \clAll$, let
$\clIntVal{K}{a}{} := K \cap \clIntVal{\clAll}{a}{}$,
$\clIntVal{K}{}{b} := K \cap \clIntVal{\clAll}{}{b}$,
$\clIntVal{K}{a}{b} := K \cap \clIntVal{\clAll}{a}{b}$.
We let also
\begin{align*}
\clEiio &:= \{ \, f \in \clAll \mid f(\vect{0}) \leq f(\vect{1}) \, \}, \\
\clEioi &:= \{ \, f \in \clAll \mid f(\vect{0}) \geq f(\vect{1}) \, \}, \\
\clNeq &:= \{ \, f \in \clAll \mid f(\vect{0}) \neq f(\vect{1}) \, \}, \\
\clEioo &:= \clNeq \cup \clII, \\
\clEiii &:= \clNeq \cup \clOO.
\end{align*}

Let $f \in \clAll^{(n)}$.
The elements of $f^{-1}(1)$ and those of $f^{-1}(0)$ are called the \emph{true points} and the \emph{false points} of $f$, respectively.
The \emph{negation} $\overline{f}$, the \emph{inner negation} $f^\mathrm{n}$, and the \emph{dual} $f^\mathrm{d}$ of $f$ are the $n$\hyp{}ary Boolean functions given by the rules
$\overline{f}(\vect{a}) := \overline{f(\vect{a})}$,
$f^\mathrm{n}(\vect{a}) := f(\overline{\vect{a}})$,
and
$f^\mathrm{d}(\vect{a}) := \overline{f(\overline{\vect{a}})}$,
for all $\vect{a} \in \{0,1\}^n$.
We can write these definitions using functional composition as
$\overline{f} := \neg(f)$,
$f^\mathrm{n} := f(\neg_1^{(n)}, \dots, \neg_n^{(n)})$,
and
$f^\mathrm{d} := \neg(f(\neg_1^{(n)}, \dots, \neg_n^{(n)}))$.
For any set $K \subseteq \clAll$,
let
$\overline{K} := \{ \, \overline{f} \mid f \in K \, \}$,
$K^\mathrm{n} := \{ \, f^\mathrm{n} \mid f \in K \, \}$,
$K^\mathrm{d} := \{ \, f^\mathrm{d} \mid f \in K \, \}$.

A function $f \in \clAll$ is \emph{self\hyp{}dual} if $f = f^\mathrm{d}$, and $f$ is \emph{reflexive} if $f = f^\mathrm{n}$.
We denote by $\clS$ the set of all self\hyp{}dual functions and by $\clRefl$ the set of all reflexive functions.

Let $f, g \in \clAll^{(n)}$.
We say that $f$ is a \emph{minorant} of $g$ or that $g$ is a \emph{majorant} of $f$, and we write $f \minorant g$, if $f(\vect{a}) \leq g(\vect{a})$ for all $\vect{a} \in \{0,1\}^n$.

\subsection{Post's lattice}
\label{subsec:Post}

The lattice of clones on $\{0,1\}$ was completely described by Post~\cite{Post}, and it is known as \emph{Post's lattice}.
It is a countably infinite lattice, and it is presented in Figure~\ref{fig:Post}.
The clones on $\{0,1\}$ are precisely the following subsets of $\clAll$:
\begin{itemize}
\item $\clAll$, $\clOX$, $\clXI$, $\clOI$,
\item $\clM := \{ \, f \in \clAll \mid \text{$f(\vect{a}) \leq f(\vect{b})$ whenever $\vect{a} \leq \vect{b}$} \, \}$ (the clone of all \emph{monotone} functions), $\clMo$, $\clMi$, $\clMc$,
\item $\clS$ (the clone of all \emph{self\hyp{}dual} functions), $\clSc$, $\clSM := \clS \cap \clM$,
\item $\clL := \clonegen{\mathord{+}, \vak{1}}$ (the clone of all \emph{linear} functions), $\clLS := \clL \cap \clS$, $\clLo$, $\clLi$, $\clLc$,
\item for $k \in \{ \, n \in \IN \mid n \geq 2 \, \} \cup \{\infty\}$,
\begin{align*}
\clUk{k} &:= \{ \, f \in \clAll \mid \forall T \subseteq f^{-1}(1) \colon \card{T} \leq k \implies \bigwedge T \neq \vect{0} \, \}, \\
\clWk{k} &:= \{ \, f \in \clAll \mid \forall F \subseteq f^{-1}(0) \colon \card{F} \leq k \implies \bigvee F \neq \vect{1} \, \}
\end{align*}
(the clone of all \emph{$1$\hyp{}separating functions of rank $k$} and the clone of all \emph{$0$\hyp{}separating functions of rank $k$}),
$\clTcUk{k}$, $\clTcWk{k}$, $\clMUk{k} := \clUk{k} \cap \clM$, $\clMWk{k} := \clWk{k} \cap \clM$, $\clMcUk{k}$, $\clMcWk{k}$,
\item $\clLambda := \clonegen{\mathord{\wedge}, \vak{0}, \vak{1}}$ (the clone of \emph{conjunctions} and constants), $\clLambdao$, $\clLambdai$, $\clLambdac$,
\item $\clV := \clonegen{\mathord{\vee}, \vak{0}, \vak{1}}$ (the clone of \emph{disjunctions} and constants), $\clVo$, $\clVi$, $\clVc$,
\item $\clOmegaOne$ (the clone of all \emph{essentially at most unary} functions, i.e., all projections, negations, and constant functions),
\item $\clIstar$ (the clone of all projections and negations),
\item $\clI$ (the clone of all projections and constant functions),
\item $\clIo$ (the clone of all projections and constant functions taking value $0$),
\item $\clIi$ (the clone of all projections and constant functions taking value $1$),
\item $\clIc$ (the clone of all projections).
\end{itemize}

\begin{figure}
\begin{center}
\scalebox{0.31}{%
\tikzstyle{every node}=[circle, draw, fill=black, scale=1, font=\LARGE]
\begin{tikzpicture}[baseline, scale=1]
   \node [label = below:$\clIc$] (Ic) at (0,-1) {};
   \node [label = right:$\clIstar$] (Istar) at (0,0.5) {};
   \node [label = right:$\clIo$] (I0) at (4.5,0.5) {};
   \node [label = left:$\clIi$] (I1) at (-4.5,0.5) {};
   \node [label = below:$\clI$] (I) at (0,2) {};
   \node [label = above:$\clOmegaOne$] (Omega1) at (0,5) {};
   \node [label = below:$\clLc$] (Lc) at (0,7.5) {};
   \node [label = right:$\clLS$] (LS) at (0,9) {};
   \node [label = right:$\clLo$] (L0) at (3,9) {};
   \node [label = left:$\clLi$] (L1) at (-3,9) {};
   \node [label = above:$\clL$] (L) at (0,10.5) {};
   \node [label = below:$\clSM\,\,$] (SM) at (0,13.5) {};
   \node [label = left:$\clSc$] (Sc) at (0,15) {};
   \node [label = above:$\clS$] (S) at (0,16.5) {};
   \node [label = below:$\clMc$] (Mc) at (0,23) {};
   \node [label = left:$\clMo\,\,$] (M0) at (2,24) {};
   \node [label = right:$\,\,\clMi$] (M1) at (-2,24) {};
   \node [label = above:$\clM\,\,$] (M) at (0,25) {};
   \node [label = below:$\clLambdac$] (Lamc) at (7.2,6.7) {};
   \node [label = left:$\clLambdai$] (Lam1) at (5,7.5) {};
   \node [label = right:$\clLambdao$] (Lam0) at (8.7,7.5) {};
   \node [label = below:$\clLambda$] (Lam) at (6.5,8.3) {};
   \node [label = left:$\clMcUk{\infty}$] (McUi) at (7.2,11.5) {};
   \node [label = left:$\clMUk{\infty}$] (MUi) at (8.7,13) {};
   \node [label = right:$\clTcUk{\infty}$] (TcUi) at (10.2,12) {};
   \node [label = right:$\clUk{\infty}$] (Ui) at (11.7,13.5) {};
   \node [label = left:$\clMcUk{3}$] (McU3) at (7.2,16) {};
   \node [label = left:$\clMUk{3}$] (MU3) at (8.7,17.5) {};
   \node [label = right:$\clTcUk{3}$] (TcU3) at (10.2,16.5) {};
   \node [label = right:$\clUk{3}$] (U3) at (11.7,18) {};
   \node [label = left:$\clMcUk{2}\,$] (McU2) at (7.2,19) {};
   \node [label = left:$\clMUk{2}\,$] (MU2) at (8.7,20.5) {};
   \node [label = right:$\clTcUk{2}$] (TcU2) at (10.2,19.5) {};
   \node [label = right:$\clUk{2}$] (U2) at (11.7,21) {};
   \node [label = below:$\clVc$] (Vc) at (-7.2,6.7) {};
   \node [label = right:$\clVo$] (V0) at (-5,7.5) {};
   \node [label = left:$\clVi$] (V1) at (-8.7,7.5) {};
   \node [label = below:$\clV$] (V) at (-6.5,8.3) {};
   \node [label = right:$\clMcWk{\infty}$] (McWi) at (-7.2,11.5) {};
   \node [label = right:$\clMWk{\infty}$] (MWi) at (-8.7,13) {};
   \node [label = left:$\clTcWk{\infty}$] (TcWi) at (-10.2,12) {};
   \node [label = left:$\clWk{\infty}$] (Wi) at (-11.7,13.5) {};
   \node [label = right:$\clMcWk{3}$] (McW3) at (-7.2,16) {};
   \node [label = right:$\clMWk{3}$] (MW3) at (-8.7,17.5) {};
   \node [label = left:$\clTcWk{3}$] (TcW3) at (-10.2,16.5) {};
   \node [label = left:$\clWk{3}$] (W3) at (-11.7,18) {};
   \node [label = right:$\,\,\clMcWk{2}$] (McW2) at (-7.2,19) {};
   \node [label = right:$\clMWk{2}$] (MW2) at (-8.7,20.5) {};
   \node [label = left:$\clTcWk{2}$] (TcW2) at (-10.2,19.5) {};
   \node [label = left:$\clWk{2}$] (W2) at (-11.7,21) {};
   \node [label = above:$\clOI$] (Tc) at (0,28) {};
   \node [label = right:$\clOX$] (T0) at (5,29.5) {};
   \node [label = left:$\clXI$] (T1) at (-5,29.5) {};
   \node [label = above:$\clAll$] (Omega) at (0,31) {};
   \draw [thick] (Ic) -- (Istar) to[out=135,in=-135] (Omega1);
   \draw [thick] (I) -- (Omega1);
   \draw [thick] (Omega1) to[out=135,in=-135] (L);
   \draw [thick] (Ic) -- (I0) -- (I);
   \draw [thick] (Ic) -- (I1) -- (I);
   \draw [thick] (Ic) to[out=128,in=-134] (Lc);
   \draw [thick] (Ic) to[out=58,in=-58] (SM);
   \draw [thick] (I0) -- (L0);
   \draw [thick] (I1) -- (L1);
   \draw [thick] (Istar) to[out=60,in=-60] (LS);
   \draw [thick] (Ic) -- (Lamc);
   \draw [thick] (I0) -- (Lam0);
   \draw [thick] (I1) -- (Lam1);
   \draw [thick] (I) -- (Lam);
   \draw [thick] (Ic) -- (Vc);
   \draw [thick] (I0) -- (V0);
   \draw [thick] (I1) -- (V1);
   \draw [thick] (I) -- (V);
   \draw [thick] (Lamc) -- (Lam0) -- (Lam);
   \draw [thick] (Lamc) -- (Lam1) -- (Lam);
   \draw [thick] (Lamc) -- (McUi);
   \draw [thick] (Lam0) -- (MUi);
   \draw [thick] (Lam1) -- (M1);
   \draw [thick] (Lam) -- (M);
   \draw [thick] (Vc) -- (V0) -- (V);
   \draw [thick] (Vc) -- (V1) -- (V);
   \draw [thick] (Vc) -- (McWi);
   \draw [thick] (V0) -- (M0);
   \draw [thick] (V1) -- (MWi);
   \draw [thick] (V) -- (M);
   \draw [thick] (McUi) -- (TcUi) -- (Ui);
   \draw [thick] (McUi) -- (MUi) -- (Ui);
   \draw [thick,loosely dashed] (McUi) -- (McU3);
   \draw [thick,loosely dashed] (MUi) -- (MU3);
   \draw [thick,loosely dashed] (TcUi) -- (TcU3);
   \draw [thick,loosely dashed] (Ui) -- (U3);
   \draw [thick] (McU3) -- (TcU3) -- (U3);
   \draw [thick] (McU3) -- (MU3) -- (U3);
   \draw [thick] (McU3) -- (McU2);
   \draw [thick] (MU3) -- (MU2);
   \draw [thick] (TcU3) -- (TcU2);
   \draw [thick] (U3) -- (U2);
   \draw [thick] (McU2) -- (TcU2) -- (U2);
   \draw [thick] (McU2) -- (MU2) -- (U2);
   \draw [thick] (McU2) -- (Mc);
   \draw [thick] (MU2) -- (M0);
   \draw [thick] (TcU2) to[out=120,in=-25] (Tc);
   \draw [thick] (U2) -- (T0);
   \draw [thick] (McWi) -- (TcWi) -- (Wi);
   \draw [thick] (McWi) -- (MWi) -- (Wi);
   \draw [thick,loosely dashed] (McWi) -- (McW3);
   \draw [thick,loosely dashed] (MWi) -- (MW3);
   \draw [thick,loosely dashed] (TcWi) -- (TcW3);
   \draw [thick,loosely dashed] (Wi) -- (W3);
   \draw [thick] (McW3) -- (TcW3) -- (W3);
   \draw [thick] (McW3) -- (MW3) -- (W3);
   \draw [thick] (McW3) -- (McW2);
   \draw [thick] (MW3) -- (MW2);
   \draw [thick] (TcW3) -- (TcW2);
   \draw [thick] (W3) -- (W2);
   \draw [thick] (McW2) -- (TcW2) -- (W2);
   \draw [thick] (McW2) -- (MW2) -- (W2);
   \draw [thick] (McW2) -- (Mc);
   \draw [thick] (MW2) -- (M1);
   \draw [thick] (TcW2) to[out=60,in=-155] (Tc);
   \draw [thick] (W2) -- (T1);
   \draw [thick] (SM) -- (McU2);
   \draw [thick] (SM) -- (McW2);
   \draw [thick] (Lc) -- (LS) -- (L);
   \draw [thick] (Lc) -- (L0) -- (L);
   \draw [thick] (Lc) -- (L1) -- (L);
   \draw [thick] (Lc) to[out=120,in=-120] (Sc);
   \draw [thick] (LS) to[out=60,in=-60] (S);
   \draw [thick] (L0) -- (T0);
   \draw [thick] (L1) -- (T1);
   \draw [thick] (L) to[out=125,in=-125] (Omega);
   \draw [thick] (SM) -- (Sc) -- (S);
   \draw [thick] (Sc) to[out=142,in=-134] (Tc);
   \draw [thick] (S) to[out=42,in=-42] (Omega);
   \draw [thick] (Mc) -- (M0) -- (M);
   \draw [thick] (Mc) -- (M1) -- (M);
   \draw [thick] (Mc) to[out=120,in=-120] (Tc);
   \draw [thick] (M0) -- (T0);
   \draw [thick] (M1) -- (T1);
   \draw [thick] (M) to[out=55,in=-55] (Omega);
   \draw [thick] (Tc) -- (T0) -- (Omega);
   \draw [thick] (Tc) -- (T1) -- (Omega);
\end{tikzpicture}
}
\end{center}
\caption{Post's lattice.}
\label{fig:Post}
\end{figure}
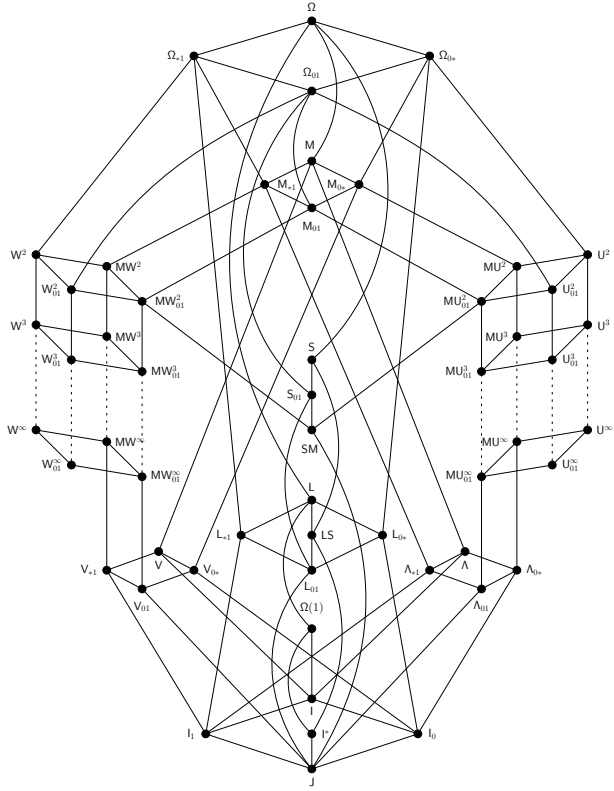

\begin{remark}
Dualization $C \mapsto C^\mathrm{d}$ is the only nontrivial automorphism of Post's lattice.
In Figure~\ref{fig:Post}, dualization corresponds to taking the mirror image along the central vertical line going through vertices $\clAll$ and $\clIc$.
\end{remark}

As a notational tool,
for sets $\mathcal{K}$ and $\mathcal{K}'$ of clones on $\{0,1\}$, we denote by $[\mathcal{K}, \mathcal{K}']$ the \emph{interval}
\[
\{ \, C \mid \text{$C$ is a clone on $\{0,1\}$ such that $K \subseteq C \subseteq K'$ for some $K \in \mathcal{K}$ and $K' \in \mathcal{K}'$} \, \}
\]
in Post's lattice.
We usually simplify this notation when one of the sets $\mathcal{K}$ and $\mathcal{K}'$ is a singleton by dropping set brackets;
thus, for example, we may write
$[\clLc,\clL]$ and $[\{\clSM, \clMcUk{\ell}, \clMcWk{\ell} \},\clAll]$
in place of 
$[\{\clLc\}, \{\clL\}]$ and $[\{\clSM, \clMcUk{\ell}, \clMcWk{\ell} \}, \{\clAll\}]$.


\subsection{Useful facts}
\label{sec:helpful}

For the needs of subsequent sections, we recall from the literature a few useful facts about clones, clonoids, and function class composition.

Although the composition of functions is associative, function class composition is not.

\begin{lemma}[{Couceiro, Foldes \cite[Associativity Lemma]{CouFol-2007,CouFol-2009}}]
\label{lem:Associativity}
Let $A$, $B$, $C$, and $D$ be arbitrary nonempty sets, and let $I \subseteq \mathcal{F}_{CD}$, $J \subseteq \mathcal{F}_{BC}$, $K \subseteq \mathcal{F}_{AB}$.
Then the following statements hold.
\begin{enumerate}[label=\upshape{(\roman*)}]
\item $(IJ)K \subseteq I(JK)$.
\item If $J$ is a minion, then $(IJ)K = I(JK)$.
\end{enumerate}
\end{lemma}

The following is a consequence of the monotonicity of function class composition.

\begin{lemma}[{\cite[Lemma~2.16]{CouLeh-Lcstability}}]
\label{lem:clonmon}
Let $C_1$ and $C'_1$ be clones on $A$ and $C_2$ and $C'_2$ clones on $B$ such that $C_1 \subseteq C'_1$ and $C_2 \subseteq C'_2$.
Then every $(C'_1,C'_2)$\hyp{}clonoid is a $(C_1,C_2)$\hyp{}clonoid.
\end{lemma}

The $(C_1,C_2)$\hyp{}clonoid generated by $F$ has a simple description in terms of function class composition.
Note the placement of brackets (cf.\ Lemma~\ref{lem:Associativity}).

\begin{lemma}[{\cite[Lemma~2.5]{Lehtonen-SM}}]
Let $F \subseteq \mathcal{F}_{AB}$, and let $C_1$ and $C_2$ be clones on $A$ and $B$, respectively.
Then $\gen[(C_1,C_2)]{F} = C_2 ( F C_1 )$.
\end{lemma}

For Boolean functions,
taking inner negations, outer negations, or duals gives an isomorphism
between the lattice of $(C_1,C_2)$\hyp{}clonoids and that of clonoids with a dualized source or target clone.

\begin{proposition}[{\cite[Lemma~5.2]{Lehtonen-SM}, \cite[Proposition~3.5]{Lehtonen-ess-lin-sem-sep}}]
\label{prop:Knid}
Let $C_1$ and $C_2$ be clones on $\{0,1\}$, and let $K \subseteq \clAll$.
Then the following statements are equivalent:
\begin{enumerate}[label=\upshape{(\alph*)}]
\item\label{prop:Knid:K} $K$ is a $(C_1,C_2)$\hyp{}clonoid.
\item\label{prop:Knid:Kn} $K^\mathrm{n}$ is a $(C_1^\mathrm{d},C_2)$\hyp{}clonoid.
\item\label{prop:Knid:Ki} $\overline{K}$ is a $(C_1,C_2^\mathrm{d})$\hyp{}clonoid.
\item\label{prop:Knid:Kd} $K^\mathrm{d}$ is a $(C_1^\mathrm{d},C_2^\mathrm{d})$\hyp{}clonoid.
\end{enumerate}
\end{proposition}


\section{Clonoids with a linear source clone and a semilattice or $0$- or $1$-separating target clone}
\label{sec:uncountable}

The main result of this paper is the following theorem that shows that for certain pairs $(C_1,C_2)$ of clones, the lattice $\closys{(C_1,C_2)}$ of $(C_1,C_2)$\hyp{}clonoids is uncountably infinite.

\begin{theorem}
\label{thm:Ulin}
For clones $C_1$ and $C_2$ such that $C_1 \subseteq K_1$ and $C_2 \subseteq K_2$ for some
\[
\begin{split}
(K_1,K_2) \in \{
   & (\clL, \clLambda), (\clLo, \clUk{\infty}), (\clLi, \clUk{\infty}), (\clLS, \clUk{\infty}), 
\\ & (\clL, \clV), (\clLo, \clWk{\infty}), (\clLi, \clWk{\infty}), (\clLS, \clWk{\infty})
\},
\end{split}
\]
there are an uncountable infinitude of $(C_1,C_2)$\hyp{}clonoids.
\end{theorem}

The remainder of this section is devoted to the proof of this result.
The proof is based on exhibiting a countably infinite family of Boolean functions with the property that distinct subsets of this family always generate distinct $(C_1,C_2)$\hyp{}clonoids, for each pair $(C_1,C_2)$ of clones prescribed in the theorem.
Because the power set of a countably infinite set is uncountable, it follows that there are an uncountable infinitude of $(C_1,C_2)$\hyp{}clonoids.

\begin{definition}
For $S \subseteq \nset{n}$, we denote by $\vect{e}_S$ the \emph{characteristic $n$\hyp{}tuple} of $S$, i.e., the tuple $(a_1, \dots, a_n) \in \{0,1\}^n$ satisfying $a_i = 1$ if and only if $i \in S$.
We write $\vect{e}_i$ for $\vect{e}_{\{i\}}$ and $\vect{e}_{ij}$ for $\vect{e}_{\{i,j\}}$.
(The number $n$ is implicit in the notation but will be clear from the context.)
\end{definition}

\begin{definition}
Let $\vect{a}, \vect{b} \in \{0,1\}^n$.
The \emph{Hamming distance} between $\vect{a}$ and $\vect{b}$ is $d(\vect{a},\vect{b}) := \{ \, i \in \nset{n} \mid a_i \neq b_i \,\}$, the number of positions in which the two tuples are different.
The \emph{Hamming weight} of $\vect{a}$ is $w(\vect{a}) := \{ \, i \in \nset{n} \mid a_i \neq 0 \, \}$, the number of non\hyp{}zero entries of $\vect{a}$.
We clearly have $w(\vect{a}) = d(\vect{a}, \vect{0})$ and $d(\vect{a}, \vect{b}) = w(\vect{a} + \vect{b})$.
Here and elsewhere, $\vect{a} + \vect{b}$ denotes the component\hyp{}wise sum of tuples $\vect{a}$ and $\vect{b}$, with addition modulo $2$.
\end{definition}

\begin{definition}
The \emph{symmetric difference} of sets $S$ and $T$ is
$S \oplus T := (S \setminus T) \cup (T \setminus S)$.
Note that for subsets $S, T \subseteq \nset{n}$, we have $\vect{e}_S + \vect{e}_T = \vect{e}_{S \oplus T}$.
\end{definition}

\begin{definition}
The set of all $1$\hyp{}element subsets of $\nset{n}$ will be denoted by $\Pone{n}$.
The set of all $1$- or $2$\hyp{}element subsets of $\nset{n}$ will be denoted by $\Pii{n}$.
\end{definition}

The following family of functions will play a key role in the proof.

\begin{definition}
For $n \in \IN^{+}$, define $\beta_n \colon \{0,1\}^n \to \{0,1\}$ by the rule $\beta_n(\vect{a}) = 1$ if and only if $w(\vect{a}) \in \{1, 2, n\}$.
In other words, $\beta_n(\vect{a}) = 1$ if and only if $\vect{a} = \vect{e}_S$ for some $S \in \Pii{n} \cup \{ \nset{n} \}$.
For any $N \subseteq \IN^{+}$, let $B_N := \{ \, \beta_n \mid n \in N \, \}$.
\end{definition}

\begin{lemma}
\label{lem:uvw-1}
Let $\vect{u}$, $\vect{v}$ be pairwise distinct tuples in $\beta_n^{-1}(1) \setminus \{\vect{1}\}$.
Then $\vect{u} + \vect{v} + \vect{1} \in \beta_n^{-1}(0)$.
\end{lemma}

\begin{proof}
Because $\vect{u}$ and $\vect{v}$ are distinct from $\vect{1}$, each of $\vect{u}$ and $\vect{v}$ has Hamming weight $1$ or $2$.
Moreover, because $\vect{u} \neq \vect{v}$, it follows that $1 \leq w(\vect{u} + \vect{v}) \leq 4$.
Since $w(\vect{u} + \vect{v} + \vect{1}) = n - w(\vect{u} + \vect{v})$, it follows that
$n - 4 \leq w(\vect{u} + \vect{v} + \vect{1}) \leq n - 1$.
Because $n > 6$, we conclude $\vect{u} + \vect{v} + \vect{1} \notin \beta_n^{-1}(1)$, i.e., $\vect{u} + \vect{v} + \vect{1} \in \beta_n^{-1}(0)$.
\end{proof}

\begin{lemma}
\label{lem:Si-2}
Let $(S_i)_{i \in \nset{m}}$ be a family of pairwise distinct elements of $\Pii{n}$, and let $T \in \Pii{n}$.
Assume that there are elements $r, s \in \nset{m}$ \textup{(}$r \neq s$\textup{)} such that $\card{S_r} = 2$, $S_r \neq T \neq S_s$, and for the set $C := S_r \oplus S_s \oplus T$, we have $\card{C} \notin \{1,2\}$.
Then one of the following three conditions holds:
\begin{enumerate}[label={\upshape(\roman*)}]
\item\label{lem:Si-2:pq} There exist $p, q \in \nset{m}$, $p \neq q$, such that 
$\card{S_p \oplus S_q \oplus C} \in \{0, 3, 4, 5, 6\}$.
\item\label{lem:Si-2:b} There exists $b \in \nset{n}$ such that $b \in S_i$ for all $i \in \nset{m}$.
\item\label{lem:Si-2:m5} $m \leq 5$.
\end{enumerate}
\end{lemma}

\begin{proof}
We have $S_r = \{a,b\}$ for some distinct $a, b \in \nset{n}$.
We consider three cases.

\textsc{Case 1:} There exists a $k \in \nset{m} \setminus \{r,s\}$ such that $\card{S_r \oplus S_k \oplus T} \notin \{1,2\}$.
Then
$S_s \oplus S_k \oplus C
= S_s \oplus S_k \oplus S_r \oplus S_s \oplus T
= S_r \oplus S_k \oplus T$.
Because $S_r, S_k, T \in \Pii{n}$, we have $\card{S_r \oplus S_k \oplus T} \leq 6$,
and we conclude that $\card{S_s \oplus S_k \oplus C} \in \{0, 3, 4, 5, 6\}$.

\textsc{Case 2:} There exists a $k \in \nset{m} \setminus \{r,s\}$ such that $\card{S_s \oplus S_k \oplus T} \notin \{1,2\}$.
The proof is similar to Case~1.

\textsc{Case 3:} Assume we are not in Case~1 nor Case~2, i.e., for all $k \in \nset{m} \setminus \{r,s\}$, we have
$\card{S_r \oplus S_k \oplus T} \in \{1,2\}$
and
$\card{S_s \oplus S_k \oplus T} \in \{1,2\}$.
We now consider four different subcases according to possible forms of the set $T$.

\textsc{Subcase 3.1:} $T = \{c,d\}$ for some distinct $c, d \in \nset{n} \setminus \{a,b\}$, and hence $S_r \oplus T = \{a, b, c, d\}$.
The sets $B \in \Pii{n}$ satisfying $\card{S_r \oplus B \oplus T} \in \{1,2\}$ and $B \neq S_r$ are
\begin{equation}
\{a,c\}, \{a,d\}, \{b,c\}, \{b,d\}, \{c,d\}.
\label{eq:3.1:B}
\end{equation}
Now, the set $S_s$ is, without loss of generality, of one of the following forms:
$\{a,c\}$, $\{a,e\}$, $\{c,e\}$, $\{e,e'\}$, $\{a\}$, $\{c\}$, $\{e\}$, where $e, e' \in \nset{n} \setminus \{a,b,c,d\}$.
We see that for each possible $S_s$, at most three of the sets $B$ listed in \eqref{eq:3.1:B} satisfy the condition
$\card{S_s \oplus B \oplus T} \in \{1,2\}$ and $B \neq S_s$, as shown in the following table.

\begin{center}
\begin{tabular}{ccl}
\toprule
$S_s$ & $S_s \oplus T$ & suitable sets $B$ \\
\midrule
$\{a,c\}$ & $\{a,d\}$ & $\{b,d\}$, $\{c,d\}$ \\
$\{a,e\}$ & $\{a,c,d,e\}$ & $\{a,c\}$, $\{a,d\}$, $\{c,d\}$ \\
$\{c,e\}$ & $\{d,e\}$ & $\{a,d\}$, $\{b,d\}$, $\{c,d\}$ \\
$\{e,e'\}$ & $\{c,d,e,e'\}$ & $\{c,d\}$ \\
$\{a\}$ & $\{a,c,d\}$ & $\{a,c\}$, $\{a,d\}$, $\{c,d\}$ \\
$\{c\}$ & $\{d\}$ & $\{a,d\}$, $\{b,d\}$, $\{c,d\}$ \\
$\{e\}$ & $\{c,d,e\}$ & $\{c,d\}$ \\
\bottomrule
\end{tabular}
\end{center}

Since each $S_k$, for $k \in \nset{m} \setminus \{r,s\}$, equals one of the suitable sets $B$,
and the $S_i$ are pairwise distinct, we must have $m - 2 \leq 3$, that is, $m \leq 5$.

\textsc{Subcase 3.2:} $T = \{c\}$ for some $c \in \nset{n} \setminus \{a,b\}$, and hence $S_r \oplus T = \{a, b, c\}$.
The sets $B \in \Pii{n}$ satisfying $\card{S_r \oplus B \oplus T} \in \{1,2\}$ and $B \neq S_r$ are
\begin{equation}
\{a,c\}, \{b,c\}, \{a\}, \{b\}, \{c\}.
\label{eq:3.2:B}
\end{equation}
Now, the set $S_s$ is, without loss of generality, of one of the following forms:
$\{a,c\}$, $\{a,e\}$, $\{b,c\}$, $\{b,e\}$, $\{c,e\}$, $\{e,e'\}$, $\{a\}$, $\{b\}$, $\{e\}$, where $e, e' \in \nset{n} \setminus \{a,b,c\}$.
We see that for each possible $S_s$, at most three of the sets $B$ listed in \eqref{eq:3.2:B} satisfy the condition
$\card{S_s \oplus B \oplus T} \in \{1,2\}$ and $B \neq S_s$, as shown in the following table.

\begin{center}
\begin{tabular}{ccl}
\toprule
$S_s$ & $S_s \oplus T$ & suitable sets $B$ \\
\midrule
$\{a,c\}$ & $\{a\}$ & $\{b\}$, $\{c\}$ \\
$\{a,e\}$ & $\{a,c,e\}$ & $\{a,c\}$, $\{a\}$, $\{c\}$ \\
$\{b.c\}$ & $\{b\}$ & $\{a\}$, $\{c\}$ \\
$\{b,e\}$ & $\{b,c,e\}$ & $\{b,c\}$, $\{b\}$, $\{c\}$ \\
$\{c,e\}$ & $\{e\}$ & $\{a\}$, $\{b\}$, $\{c\}$ \\
$\{e,e'\}$ & $\{c,e,e'\}$ & $\{c\}$ \\
$\{a\}$ & $\{a,c\}$ & $\{b,c\}$, $\{c\}$ \\
$\{b\}$ & $\{b,c\}$ & $\{a,c\}$, $\{c\}$ \\
$\{e\}$ & $\{c,e\}$ & $\{a,c\}$, $\{b,c\}$, $\{c\}$ \\
\bottomrule
\end{tabular}
\end{center}

Since each $S_k$, for $k \in \nset{m} \setminus \{r,s\}$, equals one of the suitable sets $B$,
and the $S_i$ are pairwise distinct, we must have $m - 2 \leq 3$, that is, $m \leq 5$.

\textsc{Subcase 3.3:} $T = \{a,c\}$ for some $c \in \nset{n} \setminus \{a,b\}$, and hence $S_r \oplus T = \{b, c\}$.
The sets $B \in \Pii{n}$ satisfying $\card{S_r \oplus B \oplus T} \in \{1,2\}$ and $B \neq S_r$ are precisely the following:
\begin{equation}
\text{$\{b\}$,
$\{c\}$,
$\{b,x\}$ for $x \in \nset{n} \setminus \{a, b, c\}$,
and
$\{c,y\}$ for $y \in \nset{n} \setminus \{b, c\}$.}
\label{eq:3.3:B}
\end{equation}
Now, the set $S_s$ is, without loss of generality, of one of the following forms:
$\{a,d\}$, $\{b,c\}$, $\{b,d\}$, $\{c,d\}$, $\{d,d'\}$, $\{a\}$, $\{b\}$, $\{c\}$, $\{d\}$, where $d, d' \in \nset{n} \setminus \{a,b,c\}$.
For each possible $S_s$, only some of the sets $B$ listed in \eqref{eq:3.1:B} satisfy the condition
$\card{S_s \oplus B \oplus T} \in \{1,2\}$ and $B \neq S_s$; these are shown in the following table.

\begin{center}
\begin{tabular}{ccl}
\toprule
$S_s$ & $S_s \oplus T$ & suitable sets $B$ \\
\midrule
$\{a,d\}$ & $\{c,d\}$ & $\{c\}$, $\{b,d\}$, $\{c,y\}$ for $y \in \nset{n} \setminus \{b,c,d\}$ \\
$\{b,c\}$ & $\{a,b\}$ & $\{b\}$, $\{a,c\}$, $\{b,x\}$ for $x \in \nset{n} \setminus \{a,b,c\}$ \\
$\{b,d\}$ & $\{a,b,c,d\}$ & $\{a,c\}$, $\{c,d\}$ \\
$\{c,d\}$ & $\{a,d\}$ & $\{b,d\}$, $\{a,c\}$ \\
$\{d,d'\}$ & $\{a,c,d,d'\}$ & $\{a,c\}$, $\{c,d\}$, $\{c,d'\}$ \\
$\{a\}$ & $\{c\}$ & $\{b\}$, $\{c,y\}$ for $y \in \nset{n} \setminus \{b,c\}$ \\
$\{b\}$ & $\{a,b,c\}$ & $\{c\}$, $\{a,c\}$ \\
$\{c\}$ & $\{a\}$ & $\{b\}$, $\{a,c\}$ \\
$\{d\}$ & $\{a,c,d\}$ & $\{c\}$, $\{a,c\}$, $\{c,d\}$ \\
\bottomrule
\end{tabular}
\end{center}

Recall that the sets $S_i$ are pairwise distinct and the $S_i$ for $i \in \nset{n} \setminus \{r,s\}$ are among the suitable sets $B$.
We now consider four subsubcases according to the form of the set $S_s$.

\textsc{Subsubcase 3.3.1:}
$S_s$ is one of $\{b,d\}$, $\{c,d\}$, $\{d,d'\}$, $\{b\}$, $\{c\}$, $\{d\}$.
Then there are at most three suitable sets $B$, and
we conclude that $m - 2 \leq 3$, that is, $m \leq 5$.

\textsc{Subsubcase 3.3.2:}
$S_s = \{a,d\}$ and hence $S_s \oplus T = \{c,d\}$ and $C = S_r \oplus S_s \oplus T = \{a, b, c, d\}$.
Unless $m \leq 5$ (in which case condition \ref{lem:Si-2:m5} of the lemma holds),
there exist $p, q \in \nset{m} \setminus \{r,s\}$ such that
$S_p = \{c,e\}$ and $S_q = \{c,e'\}$ for some $e, e' \in \nset{n} \setminus \{b,c,d\}$, $e \neq e'$.
Then
\[
S_p \oplus S_q \oplus C
= \{c, e\} \oplus \{c, e'\} \oplus \{a, b, c, d\}
= \{a, b, c, d, e, e'\},
\]
so $\card{S_p \oplus S_q \oplus C} = 6$.

\textsc{Subsubcase 3.3.3:}
$S_s = \{b,c\}$ and hence $S_s \oplus T = \{a,b\}$ and $C = S_r \oplus S_s \oplus T = \emptyset$.
Note that $\{a,c\}$ is the only one of the suitable sets that does not contain $b$.
Therefore, if none of the sets $S_i$ equals $\{a,c\}$, then $b \in S_i$ for all $i \in \nset{m}$,
so condition \ref{lem:Si-2:b} of the lemma holds.

Assume now that $S_p = \{a,c\}$ for some $p \in \nset{m}$.
Unless $m \leq 3$ (in which case condition \ref{lem:Si-2:m5} of the lemma holds),
there exists $q \in \nset{n} \setminus \{p, r, s\}$ such that $S_q = \{b\}$ or $S_q = \{b,e\}$ for some $e \in \nset{n} \setminus \{a, b, c\}$.
Then
$S_p \oplus S_q \oplus C$ equals either $\{a,b,c\}$ or $\{a,b,c,e\}$,
so $\card{S_p \oplus S_q \oplus C} \in \{3,4\}$.

\textsc{Subsubcase 3.3.4:}
$S_s = \{a\}$ and hence $S_s \oplus T = \{c\}$ and $C = \{a,b,c\}$.
Unless $m \leq 5$ (in which case condition \ref{lem:Si-2:m5} of the lemma holds),
there exist $p, q \in \nset{m} \setminus \{r,s\}$ such that $S_p = \{c,e\}$ and $S_q = \{c,e'\}$ for some distinct $e, e' \in \nset{n} \setminus \{a, b, c\}$.
Then $S_p \oplus S_q \oplus C = \{a,b,c,e,e'\}$,
so $\card{S_p \oplus S_q \oplus C} = 5$.

\textsc{Subcase 3.4:} $T = \{a\}$, and hence $S_r \oplus T = \{b\}$.
The sets $B \in \Pii{n}$ satisfying $\card{S_r \oplus B \oplus T} \in \{1,2\}$ and $B \neq S_r$ are precisely the following:
\begin{equation}
\text{$\{x\}$ for $x \in \nset{n} \setminus \{b\}$,
and
$\{b,y\}$ for $y \in \nset{n} \setminus \{a,b\}$.}
\label{eq:3.4:B}
\end{equation}
Now, the set $S_s$ is, without loss of generality, of one of the following forms:
$\{a,c\}$, $\{b,c\}$, $\{c,d\}$, $\{b\}$, $\{c\}$, where $c, d \in \nset{n} \setminus \{a,b\}$.
For each possible $S_s$, only some of the sets $B$ listed in \eqref{eq:3.1:B} satisfy the condition
$\card{S_s \oplus B \oplus T} \in \{1,2\}$ and $B \neq S_s$; these are shown in the following table.

\begin{center}
\begin{tabular}{ccl}
\toprule
$S_s$ & $S_s \oplus T$ & suitable sets $B$ \\
\midrule
$\{a,c\}$ & $\{c\}$ & $\{b,c\}$, $\{x\}$ for $x \in \nset{n} \setminus \{b,c\}$ \\
$\{b,c\}$ & $\{a,b,c\}$ & $\{a\}$, $\{c\}$, $\{b,c\}$ \\
$\{c,d\}$ & $\{a,c,d\}$ & $\{a\}$, $\{c\}$, $\{d\}$ \\
$\{b\}$ & $\{a,b\}$ & $\{a\}$, $\{b,y\}$ for $y \in \nset{n} \setminus \{a,b\}$ \\
$\{c\}$ & $\{a,c\}$ & $\{a\}$, $\{c\}$, $\{b,c\}$ \\
\bottomrule
\end{tabular}
\end{center}

Recall that the sets $S_i$ are pairwise distinct and the $S_i$ for $i \in \nset{n} \setminus \{r,s\}$ are among the suitable sets $B$.
We now consider three subsubcases according to the form of the set $S_s$.

\textsc{Subsubcase 3.4.1:}
$S_s$ is one of $\{b,c\}$, $\{c,d\}$, $\{c\}$.
Then there are only three suitable sets $B$, and
we conclude that $m - 2 \leq 3$, that is, $m \leq 5$.

\textsc{Subsubcase 3.4.2:}
$S_s = \{a,c\}$ and hence $S_s \oplus T = \{c\}$ and $C = \{a,b,c\}$.
Unless $m \leq 5$ (in which case condition \ref{lem:Si-2:m5} of the lemma holds),
there exist $p, q \in \nset{m} \setminus \{r,s\}$ such that $S_p = \{e\}$ and $S_q = \{e'\}$ for some distinct $e, e' \in \nset{n} \setminus \{a, b, c\}$.
Then $S_p \oplus S_q \oplus C = \{a,b,c,e,e'\}$,
and so $\card{S_p \oplus S_q \oplus C} = 5$.

\textsc{Subsubcase 3.4.3:}
$S_s = \{b\}$ and hence $S_s \oplus T = \{a,b\}$ and $C = \emptyset$.
Note that $\{a\}$ is the only one of the suitable sets $B$ that does not contain $b$.
Therefore, if none of the sets $S_i$ equals $\{a\}$, then $b \in S_i$ for all $i \in \nset{m}$,
so condition \ref{lem:Si-2:b} of the lemma holds.

Assume now that $S_p = \{a\}$ for some $p \in \nset{m}$.
Unless $m \leq 3$ (in which case condition \ref{lem:Si-2:m5} of the lemma holds),
there exists $q \in \nset{n} \setminus \{p, r, s\}$ such that $S_q = \{b,e\}$ for some $e \in \nset{n} \setminus \{a, b\}$.
Then
$S_p \oplus S_q \oplus C = \{a,b,e\}$,
and so $\card{S_p \oplus S_q \oplus C} = 3$.
\end{proof}

\begin{lemma}
\label{lem:meqn}
Let $m$ and $n$ be odd integers greater than $6$,
and assume that
$\beta_m \minorant \varphi = \beta_n \circ \mathbf{g}$, where
$\mathbf{g} = (g_1, \dots, g_n) \colon \{0,1\}^m \to \{0,1\}^n$ for some $g_1, \dots, g_n \in \clL$
such that $\mathbf{g}(\vect{0}) =: \vect{c} \in \beta_m^{-1}(0)$.
Then $m = n$.
\end{lemma}

\begin{proof}
We establish first a few properties of the mapping $\mathbf{g}$.

\begin{claim}
\label{clm:g-prop}
The map 
$\mathbf{g}$
has the following properties.
\begin{enumerate}[label={(\roman*)}]
\item\label{clm:g-prop:oddsum}
For the ternary group $\mathbf{B} = (\{0,1\}, \mathord{+_3})$ with $\mathord{+_3}(x,y,z) = x + y + z$, $\mathbf{g}$ is a homomorphism from $\mathbf{B}^m$ to $\mathbf{B}^n$, and hence for every odd natural number $2k + 1$,
\begin{equation}
\mathbf{g} \biggl( \sum_{i=1}^{2k+1} \vect{u}_i \biggr) = \sum_{i=1}^{2k+1} \mathbf{g}(\vect{u}_i),
\label{eq:g-oddsum}
\end{equation}
and for every even natural number $2k$,
\begin{equation}
\mathbf{g} \biggl( \sum_{i=1}^{2k} \vect{u}_i \biggr) = \biggl( \sum_{i=1}^{2k} \mathbf{g}(\vect{u}_i) \biggr) + \vect{c}.
\label{eq:g-evensum}
\end{equation}

\item\label{clm:g-prop:true-false}
$\mathbf{g}$ maps $\varphi^{-1}(1)$ into $\beta_n^{-1}(1)$ and $\varphi^{-1}(0)$ into $\beta_n^{-1}(0)$.

\item\label{clm:g-prop:minorant}
$\mathbf{g}$ maps $\beta_m^{-1}(1)$ into $\beta_n^{-1}(1)$.
\end{enumerate}
\end{claim}

\begin{pfclaim}[Proof of Claim~\ref{clm:g-prop}]
\ref{clm:g-prop:oddsum}
The clone $\clL$ is generated by $\mathord{+}$ and $1$.
Since $\mathord{+}$ is a homomorphism $\mathbf{B}^2 \to \mathbf{B}$ and $1$ is a homomorphism $\mathbf{B} \to \mathbf{B}$,
it follows that $\mathbf{g}$ is a homomorphism $\mathbf{B}^m \to \mathbf{B}^n$.
This means that $\mathbf{g}(\vect{u} + \vect{v} + \vect{w}) = \mathbf{g}(\vect{u}) + \mathbf{g}(\vect{v}) + \mathbf{g}(\vect{w})$.
Repeated application of this equality yields \eqref{eq:g-oddsum}.
We get \eqref{eq:g-evensum} immediately from \eqref{eq:g-oddsum} with the help of an extra term $\vect{0}$ in the sum: $\sum_{i=1}^{2k} \vect{u}_i = (\sum_{i=1}^{2k} \vect{u}_i) + \vect{0}$.
By \eqref{eq:g-oddsum}, we then get
\[
\mathbf{g} \biggl( \sum_{i=1}^{2k} \vect{u}_i \biggr)
= \mathbf{g} \biggl( \biggl( \sum_{i=1}^{2k} \vect{u}_i \biggr) + \vect{0} \biggr)
= \biggl( \sum_{i=1}^{2k} \mathbf{g}(\vect{u}_i) \biggr) + \mathbf{g}(\vect{0})
= \biggl( \sum_{i=1}^{2k} \mathbf{g}(\vect{u}_i) \biggr) + \vect{c}.
\]

\ref{clm:g-prop:true-false}
For any $\vect{a} \in \varphi^{-1}(1)$, we have $(\beta_n \circ \mathbf{g})(\vect{a}) = 1$;
hence $\mathbf{g}(\vect{a}) \in \beta_n^{-1}(1)$.
Similarly, for any $\vect{a} \in \varphi^{-1}(0)$ we have $\mathbf{g}(\vect{a}) \in \beta_n^{-1}(0)$.

\ref{clm:g-prop:minorant}
Because $\beta_m \minorant \varphi$, 
for any $\vect{a} \in \beta_m^{-1}(1)$, we have $\vect{a} \in \varphi^{-1}(1)$ and hence $\mathbf{g}(\vect{a}) \in \beta_n^{-1}(1)$.
\end{pfclaim}

\begin{claim}
\label{clm:g-prop2}
The map $\mathbf{g}$ has the following properties.
\begin{enumerate}[label={(\roman*)}]
\item\label{clm:g-prop2:inj}
For all $i, j \in \nset{m}$, if $i \neq j$, then $\mathbf{g}(\vect{e}_i)$, $\mathbf{g}(\vect{e}_j)$, and $\mathbf{g}(\vect{e}_{ij})$ are pairwise distinct.

\item\label{clm:g-prop2:inj1}
The restriction of $\mathbf{g}$ to $\Pone{m}$ is injective.

\item\label{clm:g-prop2:ijk}
For pairwise distinct $i, j, k \in \nset{m}$, we have $\mathbf{g}(\vect{e}_{ij}) + \mathbf{g}(\vect{e}_j) + \mathbf{g}(\vect{e}_k) = \mathbf{g}(\vect{e}_{ik})$.

\item\label{clm:g-prop2:1}
For all $\vect{a} \in \beta_m^{-1}(1) \setminus \{\vect{1}\}$, $\mathbf{g}(\vect{a}) \neq \vect{1}$.
\end{enumerate}
\end{claim}

\begin{pfclaim}[Proof of Claim~\ref{clm:g-prop}]
\ref{clm:g-prop2:inj}
Let $i, j \in \nset{m}$ with $i \neq j$.
Suppose, to the contrary, that $\mathbf{g}(\vect{e}_i) = \mathbf{g}(\vect{e}_j)$.
By Claim~\ref{clm:g-prop}\ref{clm:g-prop:oddsum}, we have
\begin{equation}
\mathbf{g}(\vect{0})
= \mathbf{g}(\vect{e}_i + \vect{e}_j + \vect{e}_{ij})
= \mathbf{g}(\vect{e}_i) + \mathbf{g}(\vect{e}_j) + \mathbf{g}(\vect{e}_{ij})
= \vect{0} + \mathbf{g}(\vect{e}_{ij})
= \mathbf{g}(\vect{e}_{ij}).
\label{eq:g-prop2:inj}
\end{equation}
However, $\mathbf{g}(\vect{0}) = \vect{c} \in \beta_n^{-1}(0)$ holds by our assumptions and $\mathbf{g}(\vect{e}_{ij}) \in \beta_n^{-1}(1)$ by Claim~\ref{clm:g-prop}\ref{clm:g-prop:minorant}, a contradiction.
In a similar way, we reach a contradiction if $\mathbf{g}(\vect{e}_i) = \mathbf{g}(\vect{e}_{ij})$ or $\mathbf{g}(\vect{e}_j) = \mathbf{g}(\vect{e}_{ij})$.

\ref{clm:g-prop2:inj1}
This is an immediate consequence of \ref{clm:g-prop2:inj}.

\ref{clm:g-prop2:ijk}
As in \eqref{eq:g-prop2:inj}, we have
$\mathbf{g}(\vect{0}) = \mathbf{g}(\vect{e}_i) + \mathbf{g}(\vect{e}_j) + \mathbf{g}(\vect{e}_{ij})$
and
$\mathbf{g}(\vect{0}) = \mathbf{g}(\vect{e}_i) + \mathbf{g}(\vect{e}_k) + \mathbf{g}(\vect{e}_{ik})$.
Adding these two equations and cancelling equal terms, we get
$\mathbf{g}(\vect{e}_{ij}) + \mathbf{g}(\vect{e}_j) + \mathbf{g}(\vect{e}_k) = \mathbf{g}(\vect{e}_{ik})$.

\ref{clm:g-prop2:1}
Suppose, to the contrary, that $\mathbf{g}(\vect{a}) = \vect{1}$ for some $\vect{a} \in \beta_m^{-1}(1) \setminus \{\vect{1}\}$.
Then there exist $i, j \in \nset{n}$, $i \neq j$, such that $\vect{a} = \vect{e}_i$ or $\vect{a} = \vect{e}_{ij}$.
By \ref{clm:g-prop2:inj}, $\mathbf{g}(\vect{e}_{ij}) \neq \mathbf{g}(\vect{e}_j)$.
Because $m \geq 4$, there exists $k \in \nset{m} \setminus \{i,j\}$ such that $\mathbf{g}(\vect{e}_{ij})$, $\mathbf{g}(\vect{e}_j)$, and $\mathbf{g}(\vect{e}_k)$ are pairwise distinct.
By \ref{clm:g-prop2:ijk}, we have
$\mathbf{g}(\vect{e}_{ij}) + \mathbf{g}(\vect{e}_j) + \mathbf{g}(\vect{e}_k) = \mathbf{g}(\vect{e}_{ik})$.
The summands on the left side are pairwise distinct elements of $\beta_n^{-1}(1)$ and one of them is $\vect{1}$, so by Lemma~\ref{lem:uvw-1}, the left side is an element of $\beta_n^{-1}(0)$.
By Claim~\ref{clm:g-prop}\ref{clm:g-prop:minorant}, the right side is an element of $\beta_n^{-1}(1)$.
We have reached a contradiction.
\end{pfclaim}

Suppose, to the contrary, that $m \neq n$.

By Claims~\ref{clm:g-prop}\ref{clm:g-prop:minorant} and \ref{clm:g-prop2}\ref{clm:g-prop2:inj1}, \ref{clm:g-prop2:1},
there is a family $(S_i)_{i \in \nset{m}}$ of pairwise distinct elements of $\Pii{n}$ such that $\mathbf{g}(\vect{e}_i) = \vect{e}_{S_i}$.
Assume that there is an $r \in \nset{m}$ such that $\card{S_r} = 2$.
Let $s \in \nset{m} \setminus \{r\}$, so that $\mathbf{g}(\vect{e}_s) = S_s$.
By Claims~\ref{clm:g-prop}\ref{clm:g-prop:minorant} and \ref{clm:g-prop2}\ref{clm:g-prop2:1},
we have $\mathbf{g}(\vect{e}_{rs}) = \vect{e}_T$ for some $T \in \Pii{n}$.
By Claim~\ref{clm:g-prop2}\ref{clm:g-prop2:inj}, $S_r$, $S_s$, and $T$ are pairwise distinct.
Moreover, $\mathbf{g}(\vect{0}) = \vect{c} = \vect{e}_C$ for some $C \subseteq \nset{n}$ with $\card{C} \notin \{1, 2, n\}$.
Because $\vect{0} = \vect{e}_r + \vect{e}_s + \vect{e}_{rs}$, 
it follows from Claim~\ref{clm:g-prop}\ref{clm:g-prop:oddsum} that $C = S_r \oplus S_s \oplus T$.
Note that $\card{C} = \card{S_r \oplus S_s \oplus T} \leq 6$ because each of the sets $S_r$, $S_s$, and $T$ has cardinality at most $2$.

Now, it follows from Lemma~\ref{lem:Si-2} that one of the following conditions holds:
\begin{enumerate}[label={\upshape(\roman*)}]
\item\label{cond1} There exist $p, q \in \nset{m}$, $p \neq q$, such that $\card{S_p \oplus S_q \oplus C} \in \{0, 3, 4, 5, 6\}$.
\item\label{cond2} There exists $b \in \nset{n}$ such that $b \in S_i$ for all $i \in \nset{m}$.
\item\label{cond3} $m \leq 5$.
\end{enumerate}

Condition~\ref{cond3} cannot hold, as it contradicts our assumption that $m > 6$.

Consider now the case that condition~\ref{cond1} holds.
Then
\[
\mathbf{g}(\vect{e}_{pq})
= \mathbf{g}(\vect{e}_p) + \mathbf{g}(\vect{e}_q) + \vect{c}
= \vect{e}_{S_p} + \vect{e}_{S_q} + \vect{e}_C
= \vect{e}_{S_p \oplus S_q \oplus C}.
\]
Because $\card{S_p \oplus S_q \oplus C} \in \{0, 3, 4, 5, 6\}$ and $n > 6$, we have $\vect{e}_{S_p \oplus S_q \oplus C} \notin \beta_n^{-1}(1)$.
On the other hand, $\mathbf{g}(\vect{e}_{pq}) \in \beta_n^{-1}(1)$ by Claim~\ref{clm:g-prop}\ref{clm:g-prop:minorant}.
We have reached a contradiction.

Consider now the case that condition \ref{cond2} holds.
Then, for $i \in \nset{m}$, $S_i = \{b, x_i\}$ for some $x_i \in \nset{n}$.
Because the sets $S_i$ are pairwise distinct, the elements $x_1, \dots, x_m$ must be pairwise distinct, so $m \leq n$.
Moreover, because $m \neq n$ and $m$ and $n$ are both odd, we have $m \leq n - 2$.
Because $m$ is odd, we have
\[
\mathbf{g}(\vect{1})
= \mathbf{g} \biggl( \sum_{i=1}^m \vect{e}_i \biggr)
= \sum_{i=1}^m \mathbf{g}(\vect{e}_i)
= \sum_{i=1}^m \mathbf{e}_{S_i}
= \mathbf{e}_S,
\]
where $S = \bigoplus_{i=1}^m S_i = \{b, x_1, \dots, x_m\}$.
Because $x_1, \dots, x_m$ are pairwise distinct, we have $\card{S} \in \{m, m+1\}$, depending on whether or not $b$ is one of the $x_i$.
Because $6 < m \leq n - 2$, $\vect{e}_S \notin \beta_n^{-1}(1)$.
However, by Claim~\ref{clm:g-prop}\ref{clm:g-prop:minorant}, $\mathbf{g}(\vect{1}) \in \beta_n^{-1}$, and we have again reached a contradiction.

We are left with the case that $\card{S_i} = 1$ for all $i \in \nset{m}$.
Then the $S_i$ are pairwise distinct singletons, so we must have $m \leq n$.
Moreover, because $m \neq n$ and $m$ and $n$ are both odd, we have $m \leq n - 2$.
As above,
$\mathbf{g}(\vect{1}) = \vect{e}_S$ with $S = \bigoplus_{i=1}^m S_i$, and $\card{S} = m$.
Because $6 < m \leq n - 2$, we have $\vect{e}_S \notin \beta_n^{-1}(1)$, but $\mathbf{g}(\vect{1}) \in \beta_n^{-1}(1)$, and we have reached again a contradiction.
This concludes the proof.
\end{proof}

\begin{lemma}
\label{lem:minS}
Let $N$ be the set of odd integers greater than $6$.
Let $S \subseteq N$ and $m \in N$.
Then the following statements hold.
\begin{enumerate}[label={\upshape(\roman*)}]
\item\label{lem:minS:LLambda}
$\beta_m \in \gen[(\clL,\clLambda)]{B_S}$ if and only if $m \in S$.
\item\label{lem:minS:LoUkinf}
$\beta_m \in \gen[(\clLo,\clUk{\infty})]{B_S}$ if and only if $m \in S$.
\item\label{lem:minS:LSUkinf}
$\beta_m \in \gen[(\clLS,\clUk{\infty})]{B_S}$ if and only if $m \in S$.
\end{enumerate}
\end{lemma}

\begin{proof}
If $m \in S$, then clearly $\beta_m \in B_S \subseteq \gen[(C_1,C_2)]{B_S}$ for all clones $C_1$ and $C_2$.
It remains to prove the necessity.

\ref{lem:minS:LLambda}
Assume $\beta_m \in \gen[(\clL,\clLambda)]{B_S} = \clLambda ( B_S \, \clL )$.
Then $\beta_m = \bigwedge_{i=1}^k \gamma_i$ for some $k \in \IN^{+}$,
where, for $i \in \nset{k}$,
$\gamma_i = \beta_{n_i} \circ \mathbf{g}^i$, where $n_i \in S$ and $\mathbf{g}^i = (g^i_1, \dots, g^i_{n_i})$ for some $g^i_1, \dots, g^i_{n_i} \in \clL$.
We have $\beta_m \minorant \gamma_i$ for all $i \in \nset{k}$, and for every $\vect{a} \in \beta_m^{-1}(0)$, there exists an $i \in \nset{k}$ such that $\gamma_i(\vect{a}) = 0$.
In particular, there is a $j \in \nset{k}$ such that $\gamma_j(\vect{0}) = 0$, that is, $\mathbf{g}^j(\vect{0}) \in \beta_{n_j}^{-1}(0)$.
By Lemma~\ref{lem:meqn}, $m = n_j$, so $m \in S$.

\ref{lem:minS:LoUkinf}
Assume $\beta_m \in \gen[(\clLo,\clUk{\infty})]{B_S} = \clUk{\infty} ( B_S \, \clLo )$.
Then $\beta_m = \mu ( \gamma_1, \dots, \gamma_k )$ for some $k \in \IN^{+}$,
where $\mu \in \clUk{\infty}$ and, for $i \in \nset{k}$,
$\gamma_i = \beta_{n_i} \circ \mathbf{g}^i$, where $n_i \in S$ and $\mathbf{g}^i = (g^i_1, \dots, g^i_{n_i})$ for some $g^i_1, \dots, g^i_{n_i} \in \clLo$.
Because $\mu \in \clUk{\infty}$, there exists a $p \in \nset{k}$ such that in every $\vect{a} = (a_1, \dots, a_k) \in \mu^{-1}(1)$, we have $a_p = 1$.
Consequently, for every $\vect{a} \in \beta_m^{-1}(1)$, we have $(\gamma_1, \dots, \gamma_k)(\vect{a}) \in \mu^{-1}(1)$, and hence $\gamma_p(\vect{a}) = 1$.
In other words, $\beta_m \minorant \gamma_p$.
Because $g^p_1, \dots, g^p_{n_p} \in \clLo$, we have $\mathbf{g}^p(\vect{0}) = \vect{0} \in \beta_{n_p}^{-1}(0)$.
By Lemma~\ref{lem:meqn}, $m = n_p$, so $m \in S$.

\ref{lem:minS:LSUkinf}
Assume $\beta_m \in \gen[(\clLo,\clUk{\infty})]{B_S} = \clUk{\infty} ( B_S \, \clLS )$.
Then $\beta_m = \mu ( \gamma_1, \dots, \gamma_k )$ for some $k \in \IN^{+}$,
where $\mu \in \clUk{\infty}$ and, for $i \in \nset{k}$,
$\gamma_i = \beta_{n_i} \circ \mathbf{g}^i$, where $n_i \in S$ and $\mathbf{g}^i = (g^i_1, \dots, g^i_{n_i})$ for some $g^i_1, \dots, g^i_{n_i} \in \clLS$.
Because $\mu \in \clUk{\infty}$, there exists a $p \in \nset{k}$ such that in every $\vect{a} = (a_1, \dots, a_k) \in \mu^{-1}(1)$, we have $a_p = 1$.
Consequently, for every $\vect{a} \in \beta_m^{-1}(1)$, we have $(\gamma_1, \dots, \gamma_k)(\vect{a}) \in \mu^{-1}(1)$, and hence $\gamma_p(\vect{a}) = 1$.
In other words, $\beta_m \minorant \gamma_p$.
Since $\vect{1} \in \beta_m^{-1}(1)$, we have $\mathbf{g}^p(\vect{1}) =: \vect{w} \in \beta_{n_p}^{-1}(1)$.
Because the complement of each true point of $\beta_{n_p}$ is a false point and $g^p_1, \dots, g^p_{n_p} \in \clLS$, we have $\mathbf{g}^p(\vect{0}) = \overline{\vect{w}} \in \beta_{n_p}^{-1}(0)$.
By Lemma~\ref{lem:meqn}, $m = n_p$, so $m \in S$.
\end{proof}

\begin{proof}[Proof of Theorem~\ref{thm:Ulin}]
It suffices to prove the statement for $(C_1,C_2) = (K_1,K_2)$, where
$(K_1,K_2)$ is one of the pairs
$(\clL, \clLambda)$,
$(\clLo, \clUk{\infty})$,
$(\clLS, \clUk{\infty})$.
For the remaining pairs $(K_1,K_2)$ listed in the statement of the theorem, the result follows from Proposition~\ref{prop:Knid}.
It then follows from Lemma~\ref{lem:clonmon} that for all clones $C_1$ and $C_2$ such that $C_1 \subseteq K_1$ and $C_2 \subseteq K_2$, there are an uncountable infinitude of $(C_1,C_2)$\hyp{}clonoids.

By Lemma~\ref{lem:minS},
there exists a countably infinite set $F$ of functions
with the property that for all subsets $S \subseteq F$ and for all $f \in F$,
we have $f \in \gen[(K_1,K_2)]{S}$ if and only if $f \in S$.
Consequently, for all $S, T \subseteq F$,
we have $\gen[(K_1,K_2)]{S} = \gen[(K_1,K_2)]{T}$ if and only if $S = T$.
Because the power set of $F$ is uncountable, it follows that there are an uncountable infinitude of $(K_1,K_2)$\hyp{}clonoids.
\end{proof}


\section{Examples of $(C_1,C_2)$-clonoids}
\label{sec:some}

By Theorem~\ref{thm:Ulin},
the lattice of $(C_1,C_2)$\hyp{}clonoids
is uncountable
for all pairs of clones $C_1$ and $C_2$ such that $C_1 \subseteq K_1$ and $C_2 \subseteq K_2$ for some
\[
\begin{split}
(K_1,K_2) \in \{
   & (\clL, \clLambda), (\clLo, \clUk{\infty}), (\clLi, \clUk{\infty}), (\clLS, \clUk{\infty}), 
\\ & (\clL, \clV), (\clLo, \clWk{\infty}), (\clLi, \clWk{\infty}), (\clLS, \clWk{\infty})
\}.
\end{split}
\]
It may be unfeasible to explicitly describe all such $(C_1,C_2)$\hyp{}clonoids.
We confine ourselves to a few examples of these $(C_1,C_2)$\hyp{}clonoids.
The examples are of two different types.
The first examples follow from our earlier work using the monotonicity of function class composition.
The second examples are a new, previously unseen type of clonoids that arise in an intriguing yet completely natural way from linear algebra.

\subsection{Clonoids that are implied by monotonicity of function class composition}

By Lemma~\ref{lem:clonmon}, for clones $C'_1$ and $C'_2$ with $C_1 \subseteq C'_1$ and $C_2 \subseteq C'_2$, every $(C'_1,C'_2)$\hyp{}clonoid is also a $(C_1,C_2)$\hyp{}clonoid.
Thus, taking superclones $C'_1$ and $C'_2$ for which the $(C'_1,C'_2)$\hyp{}clonoids are known, we immediately get some $(C_1,C_2)$\hyp{}clonoids.
We are now going to review earlier results that may be applied here in this way.

Many of our earlier results are presented in the following two\hyp{}part form.
First, we provide a description of all $(C_1,C_2)$\hyp{}clonoids for a certain pair $(C_1,C_2)$ of clones.
Second, we give necessary and sufficient conditions for the stability of all those $(C_1,C_2)$\hyp{}clonoids under left and right composition with clones.
From these stability conditions, we can then determine all the $(C'_1,C'_2)$\hyp{}clonoids for all superclones $C'_1 \supseteq C_1$ and $C'_2 \supseteq C_2$.
Finally, going down to subclones $K_1 \subseteq C'_1$ and $K_2 \subseteq C'_2$, every $(C'_1,C'_2)$\hyp{}clonoid is in turn a $(K_1,K_2)$\hyp{}clonoid.
In this way, we can find a few of the uncountably many $(K_1,K_2)$\hyp{}clonoids of Theorem~\ref{thm:Ulin}.

In order to apply this method, observe first that every superclone of a clone $C_1 \in [\clLc,\clL]$ is either a member of this interval or contains one of
the clones $\clSc$, $\clOX$, $\clXI$, $\clS$, and $\clAll$.
Similarly, every superclone of a clone $C_2 \in [\clLambda_c, \clUk{\infty}]$ is either a member of this interval or contains one of
the clones $\clMcUk{k}$, $\clMUk{k}$, $\clTcUk{k}$, and $\clUk{k}$ for an integer $k \geq 2$.
Now we look more carefully at clonoids with such source or target clones.

\begin{example}
The $(\clSc,\clIc)$-, $(\clS,\clIc)$-, $(\clOX,\clIc)$-, $(\clXI,\clIc)$-, and $(\clAll,\clIc)$\hyp{}clonoids are explicitly described in \cite[Theorem~7.2]{Lehtonen-discmono}.
The $(\clSc,\clVc)$\hyp{}clonoids are explicitly described in \cite[Proposition~7.7, Table~7.1, Figure~7.13]{Lehtonen-discmono},
and their stability under left and right compositions with clones is determined in \cite[Theorem~7.8, Table~7.1]{Lehtonen-discmono}.

Specifically, there are exactly 123 $(\clSc,\clVc)$\hyp{}clonoids.
They are the meet\hyp{}irreducible classes
$\clAll$,
$\clEiio \cup \clSmaj$,
$\clEioi \cup \clSmaj$,
$\clEiio$,
$\clEioi$,
$\clOX \cup \clSmaj$,
$\clXO \cup \clSmaj$,
$\clOXCI$,
$\clXOCI$,
$\clOX$,
$\clXO$,
$\clEioo \cup \clReflOO$,
$\clEioo \cup \clVako$,
$\clEioo$,
$\clSmaj \cup \clRefl$,
$\clRefl$
and all their intersections.
Because $\clLc \subseteq \clSc$, they are also $(\clLc,\clVc)$\hyp{}clonoids.
By Proposition~\ref{prop:Knid}, the duals of the above classes are $(\clLc,\clLambdac)$\hyp{}clonoids.

Here,
$\clSmin$ is the set of all minorants of self\hyp{}dual functions,
$\clSmaj$ is the set of all majorants of self\hyp{}dual functions,
and
$\clRefl$ is the set of all reflexive functions.
\end{example}

\begin{example}
The $(\clIc,\clSM)$\hyp{}clonoids are explicitly described in \cite[Theorem~4.1, Figure~3]{Lehtonen-SM}.
Necessary and sufficient conditions for their stability under right and left compositions with clones are determined in \cite[Theorem~5.1, Table~1]{Lehtonen-SM}.
From the stability conditions, we can conclude that
there are exactly 24 $(\clLc,\clMo)$\hyp{}clonoids.
They are the meet\hyp{}irreducible classes
$\clAll$,
$\clEiio$,
$\clEioi$,
$\clOXC$,
$\clIXC$,
$\clXOC$,
$\clXIC$,
$\clOX$,
$\clXO$,
$\clRefl$
and all their intersections.
Because $\clVo \subseteq \clMo$, they are also $(\clLc,\clVo)$\hyp{}clonoids.
The $(\clLc,\clMi)$\hyp{}clonoids are the duals of the $(\clLc,\clMo)$\hyp{}clonoids; because $\clLambdai \subseteq \clMi$, they are also $(\clLc,\clLambdai)$\hyp{}clonoids.
Moreover, there are exactly 11 $(\clLc,\clM)$\hyp{}clonoids; they are those $(\clLc,\clMo)$\hyp{}clonoids that are either empty or contain all constant functions.
Because $\clLambda \subseteq \clM$ and $\clV \subseteq \clM$, they are also $(\clLc,\clLambda)$- and $(\clLc,\clV)$\hyp{}clonoids.
\end{example}

\begin{example}
The $(\clIc,\clMcUk{k})$\hyp{}clonoids, for each integer $k \geq 2$, are described in \cite[Theorem~6.1]{Lehtonen-nu}.
Specifically for $k = 2$, the $(\clIc,\clMcUk{2})$\hyp{}clonoids and their stability under left and right compositions with clones are given in \cite[Table~2]{Lehtonen-nu}.
In \cite[Propositions~7.4, 7.5]{Lehtonen-nu}, necessary and sufficient conditions for stability of $(\clIc,\clMcUk{k})$\hyp{}clonoids under left and right composition with clones are provided (for right stability, the conditions are unfortunately not so explicit).

By \cite[Theorem~6.1]{Lehtonen-nu},
for an integer $k \geq 2$,
the $(\clIc,\clMcUk{k})$\hyp{}clonoids include classes of the form $\clKlik{k}{\Theta}$ for $\Theta \subseteq \clAll$, i.e., the $k$\hyp{}local closure of the so\hyp{}called minorant minion generated by $\Theta$.
(For precise definitions, see \cite[Section~4]{Lehtonen-nu}.)
By \cite[Theorem~7.4]{Lehtonen-nu}, every such class $\clKlik{k}{\Theta}$ is also a $(\clIc,\clUk{\infty})$\hyp{}clonoid.
Furthermore, \cite[Proposition~7.5]{Lehtonen-nu} gives a necessary and sufficient condition for a class $\clKlik{k}{\Theta}$ to be stable under left composition with a clone $C$; the condition is expressed in terms of the class $\Theta C$.
Applying this result for each clone $C \subseteq \clL$, we obtain some $(C,\clUk{\infty})$\hyp{}clonoids of the form $\clKlik{k}{\Theta}$.
\end{example}

\subsection{Linear algebra and clonoids}
\label{subsec:linalg}

Let $V$ be a vector space of dimension $n$ over a field $F$.
Let $S \subseteq V$.
The \emph{codimension} of $S$, denoted $\codim(S)$, is the least integer $d$ such that there exists a subspace $W$ of $V$ such that
\[
S = \bigcup_{s \in S} (s + W)
\]
($S$ is the union of some cosets of $W$) and $\codim(W) = \dim(V/W) = d$.
Note that the codimension of $S$ exists and $\codim(S) \leq n$, because every one\hyp{}element subset of $V$ is a coset of the trivial subspace $\{\vect{0}\}$ of $V$ and $\codim(\{\vect{0}\}) = \dim(V / \{\vect{0}\}) = n$.
Consequently, there exists a vector space $U$ with $\dim(U) = d$, a subset $T \subseteq U$, and a surjective linear transformation $f \colon V \to U$ such that $S = f^{-1}(T)$
(equivalently, there exists an affine subspace $A$ of a vector space $U$ with $\dim(A) = d$, a subset $T \subseteq A$, and an affine transformation $f \colon V \to U$ with range $A$ such that $S = f^{-1}(T)$).
The \emph{intersectional codimension} of $S \subseteq V$, denoted by $\icodim(S)$, is the least integer $d$ such that there exist subsets $S_i \subseteq V$ ($i \in I$) with $\codim(S_i) \leq d$ such that $S = \bigcap_{i \in I} S_i$.
Clearly $\icodim(S) \leq \codim(S)$.

\begin{lemma}
\label{lem:int:icodim}
Let $V$ be a vector space over a field $F$, and let $S_i \subseteq V$ \textup{(}$i \in I$\textup{)} with $\icodim(S_i) \leq \delta$ for all $i \in I$.
Let $S = \bigcap_{i \in I} S_i$.
Then $\icodim(S) \leq \delta$.
\end{lemma}

\begin{proof}
Because $\icodim(S_i) \leq \delta$, there exist subspaces $S_{i,j} \subseteq V$ with $\codim(S_{i,j}) \leq \delta$ ($j \in J_i$) such that $S_i = \bigcap_{j \in J_i} S_{i,j}$.
Consequently,
\[
S = \bigcap_{\substack{i \in I \\ j \in J_i}} S_{i,j},
\]
which implies that $\icodim(S) \leq \delta$.
\end{proof}

\begin{lemma}
\label{lem:aff-pre}
Let $V$ and $V'$ be vector spaces over a field $F$, and let $S \subseteq V$.
Let $f \colon V' \to V$ be an affine transformation.
Then the following statements hold.
\begin{enumerate}[label={\upshape(\roman*)}]
\item\label{lem:aff-pre:codim} $\codim(f^{-1}(S)) \leq \codim(S)$.
\item\label{lem:aff-pre:icodim} $\icodim(f^{-1}(S)) \leq \icodim(S)$.
\end{enumerate}
\end{lemma}

\begin{proof}
\ref{lem:aff-pre:codim}
Assume $\codim(S) = d$.
Then there exists a vector space $U$ of dimension $d$, a subset $T \subseteq U$, and a surjective linear transformation $L \colon V \to U$ such that $S = L^{-1}(T)$.
Now, $L \circ f \colon V' \to U$ is an affine transformation with range $H := (L \circ f)(V')$.
Because $H$ is an affine subspace of $U$ of dimension $d' \leq d$ and $f^{-1}(S) = (L \circ f)^{-1}(T) = (L \circ f)^{-1}(T \cap H)$, we conclude that $\codim(f^{-1}(S)) \leq d' \leq d$.

\ref{lem:aff-pre:icodim}
Assume that $\icodim(S) = \delta$.
Then there exist $S_i \subseteq V$ with $\codim(S_i) \leq \delta$ ($i \in I$) such that $S = \bigcap_{i \in I} S_i$.
Because
\[
f^{-1}(S)
= f^{-1} \Bigl( \bigcap_{i \in I} S_i \Bigr)
= \bigcap_{i \in I} f^{-1}(S_i)
\]
and, by part \ref{lem:aff-pre:codim}, $\codim(f^{-1}(S_i)) \leq \codim(S_i) \leq \delta$ for all $i \in I$,
we conclude that $\icodim(f^{-1}(S)) \leq \delta$.
\end{proof}

\begin{remark}
The analogue of Lemma~\ref{lem:int:icodim} does not hold for the codimension.
For example, let $V$ be a vector space of dimension $n$.
Every subspace of dimension $n-1$ has codimension $1$, but the intersection of all subspaces of dimension $n-1$ is the trivial subspace, which has codimension $n$.
\end{remark}

For $i \in \IN^{+}$, let $V_i$ be a vector space of dimension $i$ over $F$.
A family $\mathcal{K} = (\mathcal{K}_i)_{i \in \IN^{+}}$, where $\mathcal{K}_i \subseteq V_i$ is called an \emph{intersectional affine clonoid} if $\mathcal{K}$ is closed under intersections and preimages under affine transformations, that is, for all $i \in \IN^{+}$, for all $S \subseteq \mathcal{K}_i$, we have $\bigcap S \in \mathcal{K}_i$;
and for all $i, j \in \IN^{+}$, for all $S \in \mathcal{K}_i$, and for all affine transformations $f \colon V_j \to V_i$, we have $f^{-1}(S) \in \mathcal{K}_j$.

\begin{example}
\label{ex:iac}
Examples of intersectional affine clonoids include the following collections of subsets of finite\hyp{}dimensional vector spaces over $F$:
\begin{enumerate}[label=(\roman*)]
\item $(\emptyset)_{i \in \IN^{+}}$ (the empty family).
\item $(\{\emptyset\})_{i \in \IN^{+}}$ (the family of all empty subsets).
\item $(\mathcal{P}(V_i))_{i \in \IN^{+}}$ (all subsets).
\item $(\mathcal{A}_i)_{i \in \IN^{+}}$, where $\mathcal{A}_i$ is the set of all affine subspaces of $V_i$.
\item For $d \in \IN$, $(\mathcal{V}^d_i)_{i \in \IN^{+}}$, where $\mathcal{V}^d_i = \{ \, S \subseteq V_i \mid \icodim(S) \leq d \, \}$ (see Lemmata~\ref{lem:int:icodim} and \ref{lem:aff-pre}).
\end{enumerate}
\end{example}

Let us consider the above specifically for vector spaces $\{0,1\}^n$ over the two\hyp{}element field $\{0,1\}$.
The \emph{characteristic function} of a subset $S \subseteq \{0,1\}^n$ is the Boolean function $\chi_S \colon \{0,1\}^n \to \{0,1\}$ with $\chi_S(\vect{a}) = 1$ if and only if $\vect{a} \in S$.
By definition, $\chi_S^{-1}(1) = S$.
Now, a collection $\mathcal{K}_i$ of subsets of $\{0,1\}^n$ corresponds to the subset $X_{\mathcal{K}_i} = \{ \, \chi_S \mid S \in \mathcal{K}_i \, \}$.
In this way, a family $\mathcal{K} = (\mathcal{K}_i)_{i \in \IN^{+}}$ corresponds to $X_\mathcal{K} := \bigcup_{i \in \IN^{+}} X_{\mathcal{K}_i} \subseteq \clAll$.

\begin{lemma}
\label{lem:iac-LLambda}
Let $\mathcal{K} = (\mathcal{K}_i)_{i \in \IN^{+}}$, where $\mathcal{K}_i$ is a collection of subsets of $\{0,1\}^n$.
\begin{enumerate}[label={(\roman*)}]
\item $\mathcal{K}$ is closed under preimages under affine transformations if and only if $X_\mathcal{K} \clL \subseteq X_\mathcal{K}$.
\item $\mathcal{K}$ is closed under intersections if and only if $\clLambdac X_\mathcal{K} \subseteq X_\mathcal{K}$.
\item $\mathcal{K}$ is closed under intersections and $\{ \emptyset, \{0,1\}^n \} \subseteq \mathcal{K}_i$ for all $i \in \IN^{+}$ if and only if $\clLambda X_\mathcal{K} \subseteq X_\mathcal{K}$.
\end{enumerate}
\end{lemma}

\begin{proof}
Left as an easy exercise to the reader.
\end{proof}

We can now shift the concepts of codimension and intersectional codimension to Boolean functions.
Let $f \in \clAll^{(n)}$.
The \emph{codimension} and the \emph{intersectional codimension} of $f$, denoted $\codim(f)$ and $\icodim(f)$, respectively, are the codimension and the intersectional codimension, respectively, of $f^{-1}(1)$ as a subset of $\{0,1\}^n$.

In view of Example~\ref{ex:iac} and Lemma~\ref{lem:iac-LLambda}, the following sets of Boolean functions are $(\clL,\clLambdac)$\hyp{}clonoids:
\begin{enumerate}[label=(\roman*)]
\item The class $\clAff$ of all Boolean functions $f$ such that $f^{-1}(1)$ is empty or an affine subspace of $\{0,1\}^n$.
\item For $d \in \IN$, the class $\clICD{d}$ of all Boolean functions $f$ with $\icodim(f) \leq d$.
In fact, $\clICD{0} = \clVak$ and $\clICD{1} = \clAff$.
\end{enumerate}

\begin{example}
For $n \in \IN^{+}$, define the function $\mathord{\vee_n} \colon \{0,1\}^n \to \{0,1\}$, $\mathord{\vee_n}(\vect{a}) = 0$ if and only if $\vect{a} = \vect{0}$.
In other words, $\mathord{\vee_n} = \chi_S$ with $S = \{0,1\}^n \setminus \{\vect{0}\}$.
Because $\mathord{\vee}_n^{-1}(1)$ has $2^n - 1$ elements, it is a meet\hyp{}irreducible element of the power set lattice of $\{0,1\}^n$.
Moreover, since $2^n - 1$ is odd, the only possible way to express $\mathord{\vee}_n^{-1}(1)$ as the union of cosets of a subspace $W$ of $\{0,1\}^n$ is to take $W$ as the trivial subspace.
Consequently, $\icodim(\mathord{\vee_n}) = \codim(\mathord{\vee_n}) = n$.

We conclude that $\mathord{\vee_n} \in \clICD{n} \setminus \clICD{n-1}$.
Therefore, $\clICD{d} \subset \clICD{d+1}$ for all $d \in \IN$,
and we have a countably infinite ascending chain of clonoids
$\clICD{0} \subset \clICD{1} \subset \clICD{2} \subset \dots \subset \clICD{d} \subset \dots$.

In fact, every $n$\hyp{}ary Boolean function with only one false point has intersectional codimension $n$.
Denoting by $f_\vect{a}$ the function $\chi_S$ with $S = \{0,1\}^n \setminus \{\vect{a}\}$, we have that $f_\vect{a} = \mathord{\vee_n} \circ \lambda_\vect{a}$, where $\lambda_\vect{a}(\vect{x}) = \vect{x} - \vect{a}$; the function $\lambda_\vect{a}$ is an affine transformation and belongs to $\clL$.

Now, for every $g \in \clAll^{(n)} \setminus \clVaki$, we have that
\[
g
= \bigwedge_{\vect{a} \in g^{-1}(0)} f_\vect{a}
= \bigwedge_{\vect{a} \in g^{-1}(0)} (\mathord{\vee_n} \circ \lambda_\vect{a})
\in \clLambdac ( \{ \mathord{\vee_n} \} \clL )
= \gen[(\clL, \clLambdac)]{\mathord{\vee_n}}.
\]
Moreover, $\vak{1} = \mathord{\vee_n} ( \vak{1}, \dots, \vak{1} ) \in \gen[(\clL, \clLambdac)]{\mathord{\vee_n}}$.
Note also that
\[
\mathord{\vee_{n-1}} = \mathord{\vee_n}(\pr^{(n-1)}_1, \dots, \pr^{(n-1)}_{n-1}, \pr^{(n-1)}_{n-1}) \in \gen[(\clL, \clLambdac)]{\mathord{\vee_n}}.
\]
From these facts we conclude that $\bigcup_{i = 1}^n \clAll^{(i)} \subseteq \gen[(\clL, \clLambdac)]{\mathord{\vee_n}} \subseteq \clICD{n}$.
\end{example}


\section{Cardinalities of clonoid lattices of Boolean functions}
\label{sec:summary}

Putting together Theorem~\ref{thm:Ulin} and our earlier results from \cite{CouLeh-Lcstability,Lehtonen-SM,Lehtonen-nu,Lehtonen-discmono,Lehtonen-ess-lin-sem-sep},
we finally arrive at a complete classification of pairs $(C_1,C_2)$ of clones on $\{0,1\}$ according to the cardinality of the clonoid lattice $\closys{(C_1,C_2)}$.
We summarize our findings in Theorem~\ref{thm:card}, which extends Sparks's Theorem~\ref{thm:Sparks} for Boolean functions.

While Theorem~\ref{thm:card} concerns only the cardinality of $\closys{(C_1,C_2)}$,
we would like to point out that,
in the cases when the lattice $\closys{(C_1,C_2)}$ is finite or countably infinite,
descriptions of the $(C_1,C_2)$\hyp{}clonoids, often together with Hasse diagrams of $\closys{(C_1,C_2)}$, are additionally provided in the papers
\cite{CouLeh-Lcstability,Lehtonen-SM,Lehtonen-nu,Lehtonen-discmono,Lehtonen-ess-lin-sem-sep}.

We are going to need a couple of auxiliary results.
For $c \in \{0,1\}$, let $\clVaka{c}$ be the set of all constant functions in $\clAll$ taking value $c$.
For a subset $S \subseteq \{0,1\}$, let $\clVaka{S} = \bigcup_{c \in S} \clVaka{c}$.

\begin{lemma}[{\cite[Proposition~3.2]{Lehtonen-ess-lin-sem-sep}}]
\label{lem:Cvak}
Let $C_1$ and $C_2$ be clones on $\{0,1\}$, let $S \subseteq \{0,1\}$, and assume that $C_2 \cup \clVaka{S}$ is a clone on $\{0,1\}$.
Then the following statements hold.
\begin{enumerate}[label={\upshape(\roman*)}]
\item $\closys{(C_1,C_2 \cup \clVaka{S})}$ is uncountable if and only if $\closys{(C_1,C_2)}$ is uncountable.
\item $\closys{(C_1,C_2 \cup \clVaka{S})}$ is countably infinite if and only if $\closys{(C_1,C_2)}$ is countably infinite.
\item $\closys{(C_1,C_2 \cup \clVaka{S})}$ is finite if and only if $\closys{(C_1,C_2)}$ is finite.
\end{enumerate}
\end{lemma}

\begin{lemma}[{\cite[Proposition~3.11]{Lehtonen-ess-lin-sem-sep}}]
\label{lem:IcIstar}
Let $C_1$ be a subclone of $\clOX$ or of $\clXI$.
\begin{enumerate}[label={\upshape(\roman*)}]
\item $\closys{(C_1,\clIc)}$ is uncountable if and only if $\closys{(C_1,\clIstar)}$ is uncountable.
\item $\closys{(C_1,\clIc)}$ is countably infinite if and only if $\closys{(C_1,\clIstar)}$ is countably infinite.
\item $\closys{(C_1,\clIc)}$ is finite if and only if $\closys{(C_1,\clIstar)}$ is finite.
\end{enumerate}
\end{lemma}

\begin{theorem}
\label{thm:card}
Let $C_1$ and $C_2$ be clones on $\{0,1\}$.
Then the lattice $\closys{(C_1,C_2)}$ of $(C_1,C_2)$\hyp{}clonoids is
\textup{(}see Table~\ref{table:card}\textup{)}
\begin{enumerate}[label={\upshape(\roman*)}]
\item\label{thm:card:F}
finite if $C_1 \supseteq K_1$ and $C_2 \supseteq K_2$ for some
\[
\begin{split}
(K_1, K_2) \in
\{ \, & (\clIc, \clMcUk{k}) \mid 2 \leq k < \infty \, \}
\cup
\{ \, (\clIc, \clMcWk{k}) \mid 2 \leq k < \infty \, \}
\cup {}
\\ 
\{
&
(\clIc, \clSM),
(\clI, \clMcUk{\infty}),
(\clI, \clMcWk{\infty}),
(\clLambdac, \clLc),
(\clVc, \clLc),
\\ &
(\clVo, \clMcUk{\infty}),
(\clVo, \clMcWk{\infty}),
(\clLambdai, \clMcUk{\infty}),
(\clLambdai, \clMcWk{\infty}),
\\ &
(\clMc, \clLambdac),
(\clMc, \clVc),
(\clSc, \clIc)
\};
\end{split}
\]
\item\label{thm:card:C}
countably infinite if
$C_1 \in [\clIc, \clL]$ and $C_2 \in [\clLc, \clL]$
or
$C_1 \in [\clMc, \clM]$ and $C_2 \in [\clIc, \clOmegaOne]$;
\item\label{thm:card:U}
uncountable otherwise.
\end{enumerate}
\end{theorem}

\begin{table}
\begingroup\small
\begin{tabular}{r*{8}{l}}
     & & & & & $2 \leq \ell < \infty$ \\
     & & $[\clVc,\clV]$ & $[\clMcUk{\infty}, \clUk{\infty}]$ & & $[\{\clSM, \clMcUk{\ell},$ \\
     & $[\clIc,\clOmegaOne]$ & $[\clLambdac,\clLambda]$ & $[\clMcWk{\infty}, \clWk{\infty}]$ & $[\clLc,\clL]$ & $\phantom{[\{}\clMcWk{\ell} \},\clAll]$ \\
\midrule
$[\clIc, \{\clLo, \clLi, \clLS\}]$
& U\textsuperscript{8} & U\textsuperscript{9} & U\textsuperscript{9} & C\textsuperscript{6} & F\textsuperscript{1} \\
$[\clI, \clL]$
& U\textsuperscript{8} & U\textsuperscript{9} & F\textsuperscript{5} & C\textsuperscript{6} & F\textsuperscript{1} \\
$[\{\clLambdac,  \clVc, \clSM\}, \{\clUk{2}, \clWk{2}\}]$
& U\textsuperscript{12} & U\textsuperscript{11} & U\textsuperscript{10} & F\textsuperscript{3} & F\textsuperscript{1} \\
$[\clLambdai, \clLambda]$, $[\clVo, \clV]$
& U\textsuperscript{8} & U\textsuperscript{8} & F\textsuperscript{5} & F\textsuperscript{3} & F\textsuperscript{1} \\
$[\clMc, \clM]$
& C\textsuperscript{7} & F\textsuperscript{4} & F\textsuperscript{4} & F\textsuperscript{3} & F\textsuperscript{1} \\
$[\clSc, \clAll]$
& F\textsuperscript{2} & F\textsuperscript{2} & F\textsuperscript{2} & F\textsuperscript{2} & F\textsuperscript{1} \\
\end{tabular}
\endgroup

\medskip
Glossary:
F -- finite;
C -- countably infinite;
U -- uncountable

\medskip
\caption{Cardinalities of $(C_1,C_2)$\hyp{}clonoid lattices of Boolean functions. The numbers in superscript refer to the justifications given in the proof of Theorem~\ref{thm:card}.}
\label{table:card}
\end{table}

\begin{proof}
We are going to justify the entries in Table~\ref{table:card} with relevant results from earlier papers and Lemma~\ref{lem:clonmon}.
Reference to table entries is made with the numbers in superscript.

For $C_1 = \clIc$ and $C_2 \in \{\clSM\} \cup \{ \, \clMcUk{k} \mid 2 \leq k < \infty \, \} \cup \{ \, \clMcWk{k} \mid 2 \leq k < \infty \, \}$, $\closys{(C_1,C_2)}$ is finite by Sparks's Theorem~\ref{thm:Sparks}, \cite[Theorem~1.3]{Sparks-2019}.
Together with Lemma~\ref{lem:clonmon}, this gives us the entries marked with ``1''.

For $(C_1, C_2) = (\clSc, \clIc)$, $\closys{(C_1,C_2)}$ is finite by \cite[Theorem~7.2]{Lehtonen-discmono}.
Together with Lemma~\ref{lem:clonmon}, this gives us the entries marked with ``2''.

For $C_1 \in \{\clLambdac, \clVc, \clSM\}$ and $C_2 = \clLc$, $\closys{(C_1,C_2)}$ is finite by \cite[Theorem~7.1]{CouLeh-Lcstability}.
Together with Lemma~\ref{lem:clonmon}, this gives us the entries marked with ``3''.

For $C_1 = \clMc$ and $C_2 \in \{\clVc, \clLambdac\}$, $\closys{(C_1,C_2)}$ is finite by \cite[Propositions~6.12, 6.13]{Lehtonen-discmono}.
Together with Lemma~\ref{lem:clonmon}, this gives us the entries marked with ``4''.

For $C_1 \in \{\clI, \clLambdai, \clVo\}$ and $C_2 \in \{ \clMcUk{\infty}, \clMcWk{\infty} \}$, $\closys{(C_1,C_2)}$ is finite by \cite[Theorem~6.1]{Lehtonen-ess-lin-sem-sep}.
Together with Lemma~\ref{lem:clonmon}, this gives us the entries marked with ``5''.

For $C_1 \in [\clIc, \clL]$ and $C_2 \in [\clLc, \clL]$,
$\closys{(C_1,C_2)}$ is countably infinite by \cite[Theorems~6.1, 7.1]{CouLeh-Lcstability}.
This gives us the entries marked with ``6''.

For $C_1 \in [\clMc, \clM]$ and $C_2 \in [\clIc, \clOmegaOne]$,
$\closys{(C_1,C_2)}$ is countably infinite by \cite[Theorems~6.6, 6.7]{Lehtonen-discmono}.
This gives us the entry marked with ``7''.

For
\[
(C_1,C_2) \in 
\{
(\clL,\clOmegaOne),
(\clLambda,\clLambda),
(\clLambda,\clV),
(\clV,\clLambda),
(\clV,\clV),
(\clLambda,\clOmegaOne),
(\clV,\clOmegaOne)
\},
\]
$\closys{(C_1,C_2)}$ is uncountable
by \cite[Theorem~4.1]{Lehtonen-ess-lin-sem-sep}.
Together with Lemma~\ref{lem:clonmon}, this gives us the entries marked with ``8''.

For $C_1 \in \{\clLo, \clLi, \clLS\}$ and $C_2 \in \{\clUk{\infty}, \clWk{\infty}\}$,
and for $C_1 = \clL$ and $C_2 \in \{\clLambda, \clV\}$,
$\closys{(C_1,C_2)}$ is uncountable by Theorem~\ref{thm:Ulin}.
Together with Lemma~\ref{lem:clonmon}, this gives us the entries marked with ``9''.

For $C_1 \in \{\clUk{2}, \clWk{2}\}$ and $C_2 \in \{\clUk{\infty}, \clWk{\infty}\}$,
$\closys{(C_1,C_2)}$ is uncountable by \cite[Theorem~5.5]{Lehtonen-ess-lin-sem-sep}.
This gives us the entry marked with ``10''.
By Lemma~\ref{lem:clonmon}, $\closys{(C_1,C_2)}$ is uncountable for $C_1 \in \{\clUk{2}, \clWk{2}\}$ and $C_2 \in \{\clLambdao, \clVi\}$,
and then by Lemma~\ref{lem:Cvak} we have also that
$\closys{(C_1,C_2)}$ is uncountable for $C_1 \in \{\clUk{2}, \clWk{2}\}$ and $C_2 \in \{\clLambda, \clV\}$.
This gives us the entry marked with ``11''.
Similarly, by Lemma~\ref{lem:clonmon}, $\closys{(C_1,C_2)}$ is uncountable for $C_1 \in \{\clUk{2}, \clWk{2}\}$ and $C_2 = \clIc$,
and then by Lemma~\ref{lem:IcIstar} we have also that
$\closys{(C_1,C_2)}$ is uncountable for $C_1 \in \{\clUk{2}, \clWk{2}\}$ and $C_2 = \clIstar$,
and by Lemma~\ref{lem:Cvak}
$\closys{(C_1,C_2)}$ is uncountable for $C_1 \in \{\clUk{2}, \clWk{2}\}$ and $C_2 = \clOmegaOne$.
This gives us the entry marked with ``12''.
\end{proof}


\section{Open problems and final remarks}
\label{sec:remarks}

Although this paper brings to a conclusion our work on the cardinalities of lattices of $(C_1,C_2)$\hyp{}clonoids of Boolean functions,
this is obviously not the end of the story but rather a beginning.
The work reported here concerns $(C_1,C_2)$\hyp{}clonoids of Boolean functions.
This clearly calls for a generalization to clonoids on arbitrary (finite) sets.

It is known that $(C_1,C_2)$\hyp{}clonoids can be defined by relation pairs of a prescribed forms, as described by the following result.
A possible direction for future research is to find, if possible, simple relational descriptions for $(C_1,C_2)$\hyp{}clonoids (of Boolean functions); at the moment, such descriptions are known only for a few clonoids.

\begin{theorem}[{Couceiro, Foldes~\cite[Theorem~2]{CouFol-2009}}]
Let $C_1$ and $C_2$ be clones on sets $A$ and $B$, respectively, and let $F \subseteq \mathcal{F}_{AB}$.
The following conditions are equivalent.
\begin{enumerate}[label={\upshape(\roman*)}]
\item $F$ is a locally closed $(C_1,C_2)$\hyp{}clonoid.
\item $F$ is definable by some set of relation pairs $(R,S)$, where $R$ and $S$ are invariants of the clones $C_1$ and $C_2$, respectively.
\end{enumerate}
\end{theorem}

In Subsection~\ref{subsec:linalg}, we saw a connection between $(\clL,\clLambdac)$\hyp{}clonoids and
collections of subsets of vector spaces that are closed under intersections (left stability with $\clLambdac$) and preimages under affine transformations (right stability with $\clL$).
We believe the latter concept might be of interest in its own right and is potentially worth further investigation -- after all, the definition is expressed in a very simple and natural way using most basic concepts of linear algebra.
The present author is not aware of any research on this subject.


\section*{Acknowledgments}

The author would like to thank Sebastian Kreinecker for inspiring discussions.



\begin{thebibliography}{99}

\bibitem{AicMay}
    \textsc{E. Aichinger,}  \textsc{P. Mayr,}
    Finitely generated equational classes,
    \textit{J. Pure Appl.\ Algebra} \textbf{220} (2016) 2816--2827.

\bibitem{BarBulKroOpr}
    \textsc{L. Barto,}  \textsc{J. Bul\'in,}  \textsc{A. Krokhin,}  \textsc{J. Opr\v{s}al,}
    Algebraic approach to promise constraint satisfaction,
    \textit{J. ACM} \textbf{68} (2021) Art.\ 28, 66 pp.

\bibitem{Bergman}
    \textsc{C. Bergman,}
    \textit{Universal Algebra. Fundamentals and Selected Topics,}
    Pure and Applied Mathematics (Boca Raton) 301,
    CRC Press, Boca Raton, 2012.

\bibitem{BulKroOpr}
    \textsc{J. Bul\'in,}  \textsc{A. Krokhin,}  \textsc{J. Opr\v{s}al,}
    Algebraic approach to promise constraint satisfaction,
    In:
    \textit{Proceedings of the 51st Annual ACM SIGACT Symposium on the Theory of Computing \textup{(}STOC '19\textup{)},
    June 23--26, 2019, Phoenix, AZ, USA,}
    ACM, New York, NY, USA,
    pp.~602--613.

\bibitem{Cohn}
    \textsc{P. M. Cohn,}
    \textit{Universal Algebra,}
    Harper \& Row, New York, NY, 1965.

\bibitem{CouFol-2005}
    \textsc{M. Couceiro,}  \textsc{S. Foldes,}
    On closed sets of relational constraints and classes of functions closed under variable substitution,
    \textit{Algebra Universalis}\ \textbf{54} (2005) 149--165.

\bibitem{CouFol-2007}
    \textsc{M. Couceiro,}  \textsc{S. Foldes,}
    Functional equations, constraints, definability of function classes, and functions of Boolean variables,
    \textit{Acta Cybernet.}\ \textbf{18} (2007) 61--75.

\bibitem{CouFol-2009}
    \textsc{M. Couceiro,}  \textsc{S. Foldes,}
    Function classes and relational constraints stable under compositions with clones,
    \textit{Discuss.\ Math.\ Gen.\ Algebra Appl.}\ \textbf{29} (2009) 109--121.

\bibitem{CouLeh-Lcstability}
    \textsc{M. Couceiro,}  \textsc{E. Lehtonen,}
    Stability of Boolean function classes with respect to clones of linear functions,
    \textit{Order} \textbf{41} (2024) 15--64.

\bibitem{Fioravanti-AU}
    \textsc{S. Fioravanti,}
    Closed sets of finitary functions between products of finite fields of coprime order,
    \textit{Algebra Universalis} \textbf{82}(4) (2021) Art.\ 61.

\bibitem{Fioravanti-IJAC}
    \textsc{S. Fioravanti,}
    Expansions of abelian square\hyp{}free groups,
    \textit{Internat.\ J. Algebra Comput.}\ \textbf{31}(4) (2021) 623--638.

\bibitem{Kreinecker}
    \textsc{S. Kreinecker,}
    Closed function sets on groups of prime order,
    \textit{J. Mult.\hyp{}Valued Logic Soft Comput.}\ \textbf{33} (2019) 51--74.

\bibitem{Lehtonen-SM}
    \textsc{E. Lehtonen,}
    Majority\hyp{}closed minions of Boolean functions,
    \textit{Algebra Universalis} \textbf{85} (2024) Art.\ 6.

\bibitem{Lehtonen-nu}
    \textsc{E. Lehtonen,}
    Near\hyp{}unanimity\hyp{}closed minions of Boolean functions,
    \textit{Algebra Universalis} \textbf{86} (2025) Art.\ 2.
    
\bibitem{Lehtonen-discmono}
    \textsc{E. Lehtonen,}
    Clonoids of Boolean functions with a monotone or discriminator source clone,
    arXiv:2405.01164.

\bibitem{Lehtonen-ess-lin-sem-sep}
    \textsc{E. Lehtonen,}
    Clonoids of Boolean functions with essentially unary, linear, semilattice, or $0$- or $1$\hyp{}separating source and target clones,
    arXiv:2412.01107.

\bibitem{MayWyn}
    \textsc{P. Mayr,}  \textsc{P. Wynne,}
    Clonoids between modules,
    \textit{Internat.\ J. Algebra Comput.}\ \textbf{34} (2024) 543--570.

\bibitem{Pippenger}
    \textsc{N. Pippenger,}
    Galois theory for minors of finite functions,
    \textit{Discrete Math.}\ \textbf{254} (2002) 405--419.

\bibitem{Post}
    \textsc{E. L. Post,}
    \textit{The Two-Valued Iterative Systems of Mathematical Logic,}
    Annals of Mathematics Studies, no.\ 5,
    Princeton University Press, Princeton, 1941.

\bibitem{Sparks-2019}
    \textsc{A. Sparks,}
    On the number of clonoids,
    \textit{Algebra Universalis} \textbf{80}(4) (2019) Art.\ 53, 10 pp.

\bibitem{Szendrei}
    \textsc{\'A. Szendrei,}
    \textit{Clones in Universal Algebra,}
    S\'eminaire de Math\'ematiques Sup\'erieures, no.\ 99,
    Les Presses de l'Universit\'e de Montr\'eal,
    Montr\'eal, 1986.

\end{thebibliography}
\end{document}